\documentclass[12pt]{amsart}
\usepackage{ifpdf}
\ifpdf\usepackage[colorlinks,pagebackref]{hyperref} 
\else\usepackage[hypertex,pagebackref]{hyperref}\fi 
\usepackage{amssymb}
\usepackage[curve]{xypic}
\newcommand{\comment}[1]{\marginpar{\sffamily{\raggedright\noindent\tiny #1
   \par}\normalfont}}
\newcommand{\Omit}[1]{\begin{tiny}#1\end{tiny}}
\renewcommand{\Omit}[1]{}
\renewcommand{\comment}[1]{}

\newbox\mybox
\def\overtag#1#2#3{\setbox\mybox\hbox{$#1$}\hbox to
  0pt{\vbox to 0pt{\vglue-#3\vglue-\ht\mybox\hbox to \wd\mybox
      {\hss$\ss#2$\hss}\vss}\hss}\box\mybox}
\def\undertag#1#2#3{\setbox\mybox\hbox{$#1$}\hbox to 0pt{\vbox to
    0pt{\vglue#3\vglue\ht\mybox\hbox to \wd\mybox
      {\hss$\ss#2$\hss}\vss}\hss}\box\mybox}
\def\lefttag#1#2#3{\hbox to 0pt{\vbox to 0pt{\vss\hbox to
      0pt{\hss$\ss#2$\hskip#3}\vss}}#1}
\def\righttag#1#2#3{\hbox to 0pt{\vbox to 0pt{\vss\hbox to
      0pt{\hskip#3$\ss#2$\hss}\vss}}#1}
\let\ss\scriptstyle
\def\notags{\def\overtag##1##2##3{##1}
  \def\undertag##1##2##3{##1}\def\lefttag##1##2##3{##1}
  \def\righttag##1##2##3{##1}}
\def\Dot{\lower.2pc\hbox to 2.5pt{\hss$\bullet$\hss}}
\def\Circ{\lower.2pc\hbox to 2.5pt{\hss$\circ$\hss}}
\def\Vdots{\raise5pt\hbox{$\vdots$}}
\def\splicediag#1#2{\xymatrix@R=#1pt@C=#2pt@M=0pt@W=0pt@H=0pt}

\renewcommand\frame[2][3pt]{\hbox{$\vcenter{\hbox{\vrule\vbox
{\hrule\kern#1\hbox{\kern#1$#2$\kern#1}\kern#1\hrule}\vrule}}$}}

\def\ab#1{\widetilde #1^{ab}}

\newcommand\UAC{universal abelian cover}
\newcommand\lineto{\ar@{-}}
\newcommand\dashto{\ar@{--}}
\newcommand\dotto{\ar@{.}}
\newcommand{\C}{\mathbb C}

\newcommand{\N}{\mathbb N}
\newcommand{\Z}{\mathbb Z}
\newcommand{\Q}{\mathbb Q}
\newcommand{\calQ}{\mathcal Q}
\newcommand{\calS}{\mathcal S}
\newcommand{\calD}{\mathcal D}
\newcommand{\calJ}{\mathcal J}
\newcommand{\calN}{\mathcal N}
\newcommand{\m}{\mathfrak m}
\newcommand{\Hom}{\operatorname{Hom}}
\newtheorem*{ECtheorem}{End Curve Theorem}
\newtheorem*{claim}{Claim}
\newtheorem*{theorem*}{Theorem}
\newtheorem{theorem}{Theorem}[section]
\newtheorem{proposition}[theorem]{Proposition}
\newtheorem{lemma}[theorem]{Lemma}
\newtheorem{corollary}[theorem]{Corollary}
\theoremstyle{definition}
\newtheorem{example}[theorem]{Example}
\newtheorem*{example*}{Example}
\newtheorem{definition}[theorem]{Definition}
\newtheorem{remark}[theorem]{Remark}
\newtheorem*{remark*}{Remark}
\begin{document}
\title[The End Curve Theorem for normal complex surface singularities]
{The End Curve Theorem for normal complex surface singularities}
\author{Walter D. Neumann} \thanks{Research supported under NSF grant
  no.\ DMS-0456227 and NSA grant no.\ H98230-06-1-011} \address{Department of
  Mathematics\\Barnard College, Columbia University\\New York, NY
  10027} \email{neumann@math.columbia.edu} \author{Jonathan Wahl}
\thanks{Research supported under NSA grant no.\ FA9550-06-1-0063}
\address{Department of Mathematics\\The University of North
  Carolina\\Chapel Hill, NC 27599-3250} \email{jmwahl@email.unc.edu}
\begin{abstract} We prove the ``End Curve Theorem,'' which states that
  a normal surface singularity $(X,o)$ with rational homology sphere
  link $\Sigma$ is a splice-quotient singularity if and only if it has
  an end curve function for each leaf of a good resolution tree.

  An ``end-curve function'' is an analytic function $(X,o)\to (\C,0)$
  whose zero set intersects $\Sigma$ in the knot given by a meridian
  curve of the exceptional curve corresponding to the given leaf.

  A ``splice-quotient singularity'' $(X,o)$ is described by giving an
  explicit set of equations describing its universal abelian cover as
  a complete intersection in $\C^t$, where $t$ is the number of leaves
  in the resolution graph for $(X,o)$, together with an explicit
  description of the covering transformation group.

 Among the immediate consequences of the End Curve Theorem are
the previously known results: $(X,o)$ is a splice quotient if it
is weighted homogeneous (Neumann 1981), or rational or minimally
elliptic (Okuma 2005).
\end{abstract}
\maketitle

We consider normal surface singularities whose links are rational
homology spheres ($\Q$HS for short).  The $\Q$HS condition is
equivalent to the condition that the resolution graph $\Gamma$ of a
minimal good resolution be a \emph{rational tree}, i.e., $\Gamma$
is a tree and all exceptional curves are genus zero.

Among singularities with $\Q$HS links, splice-quotient singularities
are a broad generalization of weighted homogeneous singularities. We
recall their definition briefly here and in more detail in Section
\ref{sec:splice quotients}. Full details can be found in \cite{nw1}.

Recall first that the topology of a normal complex surface singularity
$(X,o)$ is determined by and determines the minimal resolution graph
$\Gamma$. Let $t$ be the number of leaves of $\Gamma$. For $i=1,\dots,
t$, we associate the coordinate function $x_i$ of $\C^t$ to the $i$--th
leaf. This leads to a natural action of the ``discriminant group''
$D=H_1(\Sigma)$ by diagonal matrices on $\C^t$ (see Section
\ref{sec:splice quotients}).

Under two (weak) conditions on $\Gamma$, called the ``semigroup"
and ``congruence" conditions, one can write down an explicit set
of $t-2$ equations in the variables $x_i$, which defines an
isolated complete intersection singularity $(V,0)$ and which is
invariant under the action of $D$. Moreover the resulting action
of $D$ on $V$ is free away from $0$, and $(X,o)=(V,0)/D$ is a
normal surface singularity whose minimal good resolution graph is
$\Gamma$. This $(X,o)$ is what we call a \emph{splice quotient
singularity}. Since the covering transformation group for the
covering map $V\to X$ is $D=H_1(X-\{o\})=\pi_1(X-\{o\})^{ab}$, the
covering $(V,0)\to (X,o)$ (branched only at the singular points)
may be called the \emph{universal abelian cover} 
of
$(X,o)$.  In particular, for a splice quotient singularity, one
can write down explicit equations for the \UAC{} just from the
resolution graph, i.e., from the topology of the link.

The link $\Sigma$ of the singularity $(X,o)$ can be expressed as the
boundary of a plumbed regular
neighborhood $N$ of the exceptional divisor $E=E_1\cup\dots\cup E_n$
in the minimal good resolution $\tilde X$ of $(X,o)$.  Then each
meridian curve of an $E_i$ gives a knot $K_i$ in $\Sigma$. A
``meridian curve'' means the boundary of a small transverse disk to
the exceptional divisor $E_i$. If $E_i$ is the exceptional curve
corresponding to a leaf of $\Gamma$ we call $K_i$ an \emph{end knot}.
A (germ of a) smooth complex curve on $\tilde X$ which intersects $E$
transversally on such a leaf curve (and hence which cuts out an
end-knot on $\Sigma$) is called an \emph{end curve}; we also use this
name for the image curve in $X$. 

If $(X,o)$ is a splice-quotient singularity as described above, then
some power $z_i=x_i^d$ of the coordinate function $x_i$ on $V$ is well
defined on $X=V/D$. The zero set in $\Sigma$ resp.\ $X$ of $z_i$ is
the end knot resp.\ end curve corresponding to the $i$--th leaf of
$\Gamma$ (the degree of vanishing may be $> 1$). We say that the end
knot or end curve is \emph{cut out} by the function $z_i$ and that
$z_i$ is an \emph{end curve function}.

Our main result is\comment{WDN: There is certainly a rationale for
  changing the name to ``End Curves Theorem'' or ``End-Curves
  Theorem'' but since we have referred to it elsewhere as ``End Curve
  Theorem'' and others have used this name too, I think we must stick
  with it.}

\begin{ECtheorem}
  Let $(X,o)$ be a normal surface singularity with $\Q$HS link
  $\Sigma$.  Suppose that for each leaf of the resolution diagram
  $\Gamma$ there exists a corresponding end curve function $z_i\colon
  (V,o) \to (\C,0)$ which cuts out an end knot $K_i\subset \Sigma$ (or
  end curve) for that leaf. Then $(X,o)$ is a splice quotient
  singularity and a choice of a suitable root $x_i$ of $z_i$ for each
  $i$ gives coordinates for the splice quotient description.
\end{ECtheorem}

An immediate corollary (conjectured in \cite{nw1} and first proved by
Okuma \cite{okuma2}) is that rational singularities and most minimally
elliptic singularities (the few with non--$\Q$HS link must be
excluded) are splice-quotients. Another direct corollary is the result
of \cite{neumann83}, that a weighted homogeneous singularity with
$\Q$HS link has \UAC{} a Brieskorn complete
intersection.  The special case of the End Curve Theorem, when the
link is an \emph{integral} homology sphere (so that $D$ is trivial),
was proved in our earlier paper \cite{nw2}.

We first proved the End Curve Theorem in summer of 2005, but it has
taken a while to write up in what we hope is an understandable
form. In the meantime, Okuma resp.\ N\'emethi and Okuma in
\cite{okuma3, nemethi-okuma1, nemethi-okuma2} (see also Braun and
N\'emethi \cite{braun-nemethi}) have used this to compute the
geometric genus $p_g$ of any splice-quotient, and to prove for
splice-quotients the Casson invariant conjecture \cite{neumann-wahl90}
for singularities with $\Z$HS links (in which case $D=\{1\}$ so
$V=X$), as well as the N\'emethi-Nicolaescu extension
\cite{nemethi-nicolaescu} of the Casson Invariant Conjecture to
singularities with $\Q$HS links. 

These results of N\'emethi and Okuma give topological interpretations
of analytic invariants; this is analogous to the fact that for
rational singularities some of the important analytic invariants are
topologically determined. As happens for rational singularities, the
set of resolution graphs that belong to splice quotient singularities
is closed under the operations of taking subgraphs and of decreasing
the intersection weight at any vertex \cite{nemethi-okuma2}.  It is
worth noting, however, that rational singularities did not have
explicit analytic descriptions before splice quotients were
discovered; even the fact that their \UAC{}s are
complete intersections was unexpected until it was conjectured in
\cite{nw1} (see also \cite{okuma1}).

Of course, unlike rationality, the property of being a splice-quotient
is not topologically determined---for example, splice quotients, as
quotients of Gorenstein singularities, are necessarily
$\Q$--Gorenstein, which is generally a very special property within a
topological type.  Even more, ``equisingular deformations'' of very
simple splice quotients need not be of this type (see Example
\ref{ex:sell}\comment{WDN-reference provided}).

We once over-optimistically conjectured that $\Q$--Gorenstein
singularities with $\Q$HS links would have complete intersection
\UAC{}s, and although this is false in general \cite{nemethi et
al}, we see it is true for a large class of singularities. There
is a natural arithmetic analog. A standard ``dictionary'' that
developed out of proposals of Mazur and others pairs 3--manifolds
with number fields, knots with primes, and so on. A natural analog
of \UAC{}s of $Q$HS links belonging to complete intersections
would be that the ring of integers of the Hilbert class field of a
number field $K$ be a complete intersection over $\Z$. This is
true, proved by de Smit and Lenstra \cite{de smit}. The analogy
between splice singularities and Hilbert class fields is enticing,
since it is a significant open problem to compute Hilbert class
fields, while the explicit splice singularity description is
easily computed from the resolution diagram.\comment{``In particular,
D.Zagier has suggested a possible number field analogue of the
bilinear pairing on our topological discriminant groups. He didn't
shoot them down completely, but I don't think there was anything
definite enough that we can cite it. `` 
WDN Zagier suggested a candidate JMW Are you sure?  He mentioned a
few ideas, but I thought he shot them all down} \comment{JMW to
add description of contents of the paper, once it's
  complete. WDN: added a draft; note that we have a summary of
  content also in section \ref{sec:overview}}

\medskip
We summarise the proof of the End Curve Theorem in Section
\ref{sec:overview} after first recalling the theory of splice quotient
singularities in Section \ref{sec:splice quotients}. We complete the
proof in Section \ref{sec:proof}.  Some applications and examples are
discussed in the final section \ref{sec:examples}.

Some of the ingredients in our proof could be of independent
interest. We need an extension to the equivariant reducible case
of the theory of numerical semigroups and monomial curves
developed by Delorme, Herzog, Kunz, Watanabe and the authors
\cite{delorme, herzog,
  herzog-kunz, nw2, watanabe}. The necessary parts of this theory are
developed in sections \ref{sec:D-curves}--\ref{sec:nu}. Some
topological results about knots in $\Q$--homology spheres and their
linking numbers and Milnor numbers are collated in Section
\ref{sec:top}.

\section{Splice quotient singularities}\label{sec:splice quotients}
We recall here the detailed construction of splice-quotient singularities. For
full details see \cite{nw1}.

Let $(\bar{X},o)\subset (\C^N,o)$ be a normal surface
singularity whose link $\Sigma =\bar{X}\cap S^{2N-1}_\epsilon$
is a $\Q$HS. Equivalently the minimal good resolution resolves the
singularity by a tree of rational curves. Let $\Gamma$ be the
resolution graph.  In some cases we can construct directly from
$\Gamma$ singularities which have the same link as $\bar{X}$ (but
might well be analytically distinct).

We denote by $A(\Gamma)$ the intersection matrix of the exceptional
divisor (we say ``intersection matrix of $\Gamma$''); this is the
negative definite matrix whose diagonal entries are the weights of the
vertices of $\Gamma$ and whose off-diagonal entries are $1$ or $0$
according as corresponding vertices of $\Gamma$ are connected by an
edge or not.  The \emph{discriminant group} $D(\Gamma)$ is the
cokernel of $A(\Gamma)\colon \Z^n\to \Z^n$. There is a canonical
isomorphism $D(\Gamma)\cong H_1(\Sigma;\Z)$ (if $\Sigma$ were not a
$\Q$HS one would have $D(\Gamma)\cong
\operatorname{Tor}H_1(\Sigma;\Z)$). The order
$|D|=\det(-A(\Gamma))$ of $D(\Gamma)$ is called the
\emph{discriminant} of $\Gamma$.

\subsection{Splice diagram}
\label{subsec:splice diagram} We shall denote by $\Delta$ the
splice diagram corresponding to $\Gamma$. We recall its
construction. If one removes from $\Gamma$ a vertex $v$ and its
adjacent edges then $\Gamma$ breaks into subgraphs $\Gamma_{vj}$,
$j=1,\dots,\delta_v$, where $\delta_v$ is the valency of $v$. We
weight each outgoing edge at $v$ by the discriminant of the
corresponding subgraph; these are the ``splice diagram weights''
(the reader may wish to refer to the illustrative example in
subsection \ref{subsec:example}). The graph $\Gamma$ with all
splice diagram weights added and with the self-intersection
weights deleted is called the \emph{maximal splice diagram}. One
can still recover $\Gamma$ from it. If one now drops the splice
diagram weights around vertices of valency $\le 2$ and then
suppresses all valency 2 vertices to get a diagram with only
leaves (valency $1$) and nodes (valency $\ge 3$) one gets the
\emph{splice diagram} $\Delta$.\comment{JMW: what happens for a
cyclic quotient singularity? WDN: it works, but its a bit trivial
-- the two-vertex splice diagram corresponds to univ abelian cover
$S^3$. JMW:but shouldn't we explicitly exclude this case? WDN: Why? it
fits in the general theory.} The
splice diagram $\Delta$ no longer determines $\Gamma$ in general.

For the purpose of this paper it is convenient to have a version of
the splice diagram in which we do not discard the splice diagram
weights at leaves. We call this the \emph{splice diagram with leaf
  weights} and denote it also by $\Delta$.

\begin{definition}\label{def:linking}
  For two vertices $v$ and $w$ of $\Gamma$ the \emph{linking
    number}.  $\ell_{vw}$ is the product of splice
  diagram weights adjacent to but not on the shortest path from $v$ to
  $w$ in $\Gamma$. If $v=w$ this means the product of splice diagram
  weights at $v$.  (The name comes from the fact that $\ell_{vw}$ is
    $|D|$ times the linking number of the knots in $\Sigma$
    corresponding to $v$ and $w$, see Proposition
    \ref{prop:linking}.)
\end{definition}
The matrix $(\ell_{vw})$ is the adjoint of $-A(\Gamma)$
(\cite{eisenbud-neumann} Lemma 20.2):
$$ (\ell_{vw}) = \operatorname{Adj}(-A(\Gamma))=-|D|A(\Gamma)^{-1}$$
Note that for vertices $v$ and $w$ of $\Delta$, $\ell_{vw}$ can be
computed using only weights of $\Delta$, except that
leaf weights are also needed if $v=w$ is a leaf.

\subsection{Action of the discriminant group on $\C^t$}
\label{subsec:action}
Let $v_i$, $i=1,\dots,t$ be the leaves of $\Gamma$ or $\Delta$ and
associate a coordinate $Y_i$ of $\C^t$ with each leaf\comment{JMW--we used
$x_i$ in the third paragraph of the Introduction. WDN: I know -- I
thought about this when uniformizing the notation, but swung for
leaving $x_i$ in the intro}. Since
$D=\Z^n/A(\Gamma)\Z^n$ with $\Z^n=\Z^{\operatorname{vert(\Gamma)}}$,
each vertex $v$ of $\Gamma$ determines an element $e_v\in D$. There is
a non-degenerate $\Q/\Z$--valued bilinear form on $D$ satisfying
$$e_v\cdot e_w=-\ell_{vw}/|D|\,,$$
the $vw$--entry of $A(\Gamma)^{-1}$.

We get an action of $D$ on $\C^t$ by letting the element $e\in D$ act
via the diagonal matrix
$$\operatorname{diag}(e^{2\pi i \,e\cdot e_{v_1}},\dots,
e^{2\pi i \,e\cdot e_{v_t}})\,.$$
\noindent The elements $e_{v_i}$, $i=1,\dots,t$ generate $D$ (in fact
any $t-1$ of them do, see \cite{nw1} Proposition 5.1).  We thus only
need the splice diagram with leaf weights $\Delta$ to determine this
action.

\subsection{Splice equations}
In contrast to the action of $D$ on $\C^t$, only the splice diagram
$\Delta$ and not leaf weights are needed to discuss splice
equations. We will write down $t-2$ equations in the variables
$Y_1,\dots,Y_t$, grouped into $\delta_v-2$ equations for each node $v$
of $\Delta$. These $\delta_v-2$ equations are weighted homogeneous
with respect to weights determined by $v$. We first describe these
weights.

Fix a node $v$ of $\Delta$. The \emph{$v$-weight} of $Y_i$ is
$\ell_{vv_i}$. We will write down equations of total weight
$\ell_{vv}$. Number the outgoing edges at $v$ by
$j=1,\dots,\delta_v$. For each $j$, a monomial $M_{vj}$ of total
weight $\ell_{vv}$, using only the variables $Y_j$ that are beyond the
outgoing edge $j$ from $v$, is called an \emph{admissible monomial}.
The existence of admissible monomials for every edge at every node is
the \emph{semigroup condition} of \cite{nw1}. Assuming this condition,
choose one admissible monomial $M_{vj}$ for each outgoing edge at
$v$. Then \emph{splice diagram equations} for the node $v$ consist of
equations of the form
$$\sum_{j=1}^{\delta_v} a_{vij} M_{vj} +H_{vi}=0\,\quad
i=1,\dots,\delta_v-2,$$
where
\begin{itemize}
\item all maximal minors of the $(\delta_v-2)\times \delta_v$ matrix
  $(a_{vij})$ have full rank;
\item $H_{vi}$ is an optional extra summand in terms of monomials of
  $v$--weight $>\ell_{vv}$ (most generally, a convergent power series
  in such monomials).
\end{itemize}
Choosing splice diagram equations for each node gives exactly $t-2$
equations, called a \emph{system of splice diagram equations}. In Theorem 2.6 of
\cite{nw1}, it is shown that they determine a 2-dimensional complete
intersection $V$ with isolated singularity at the origin.
\begin{claim}
  The link of this singularity has the topology of the universal
  abelian cover of the singularity link determined by $\Gamma$. In
  particular, this topology is determined by the splice diagram
  $\Delta$ alone.
\end{claim}
\noindent For $\Gamma$ that admit splice-quotient singularities (i.e.,
the equations can be chosen $D$--equivariantly), this is in
\cite{nw1}. The second sentence has been proved in full generality by
Helge Pedersen \cite{pedersen}, but a complete proof of the first
sentence for general $\Gamma$ has not yet been written
up, so it should be considered conjectural.

\subsection{Splice-quotient singularities}

Suppose now that we can choose a system of splice diagram
equations as above, which are additionally equivariant with
respect to the action of $D$; this is a\comment{WDN Deleted
  ``second'' again; it is the
  \emph{first} condition on $\Gamma$; the other was on $\Delta$} combinatorial
condition on $\Gamma$, called the \emph{congruence condition}
(\cite{neumann-wahl02, nw1}). Then
\begin{theorem*}[\cite{nw1}]
  $D$ acts freely on $V-\{0\}$ and the quotient $(X,o):=(V,0)/D$ is a
  normal surface singularity whose resolution graph is $\Gamma$; moreover,
  $(V,0)\to (X,o)$ is the \UAC. We call $(X,o)$ a
  \emph{splice-quotient singularity}.\comment{J: do we really want
  different size "oh" for the origins? W: I don't really care. They
  are ``oh'' and ``zero''$\in \C^t$}
\end{theorem*}

Theorem 10.1 of \cite{nw1} says that the class of splice-quotient
singularities is natural, in the sense that it does not depend on the
choice of which admissible monomials $M_{vj}$ one chooses to use (so
long as they are chosen equivariantly for the action of $D$). A change
in choice can be absorbed in the extra higher order summands of the
splice equations.

\subsection{Reduced weights}
The $v$--weights of the $Y_i$ used to define splice equations may have
a common factor, so in practice one should use reduced $v$-weights
that divide out this common factor. Precisely, if $v$ is a node the
\emph{reduced $v$--weight} of the variable $Y_j$ is the $v$-weight
$\ell_{vv_j}$ of $Y_j$ divided by the GCD of the $v$--weights of all the
$Y_i$'s.

\subsection{Example}\label{subsec:example}

Consider the resolution graph
$$\splicediag{15}{40}{
&&&&&\overtag{\Circ}{-2}{6pt}\\
\overtag{\Circ}{-2}{6pt}\lineto[dr]&&&&\overtag{\Circ}{-2}{6pt}\lineto[ur]\\
\hbox to 0pt{\hss$\Gamma=$\quad}&\overtag{\Circ}{-2}{6pt}\lineto[r]&\overtag{\Circ}{-3}{6pt}\lineto[r]&
  \overtag{\Circ}{-2}{6pt}\lineto[ur]\lineto[dr]\\
\overtag{\Circ}{-3}{6pt}\lineto[ur]&&&&\overtag{\Circ}{-2}{6pt}\\&}$$
Its maximal splice diagram is
$$\notags\splicediag{15}{45}{
  &&&&&\overtag{\Circ}{-2}{6pt}\\
  \overtag{\Circ}{-2}{6pt}\lineto[dr]^(.25){30}^(.75){2}
  &&&&\overtag{\Circ}{-2}{6pt}\lineto[ur]^(.25){2}^(.75){32}\\
  &\overtag{\Circ}{-2}{6pt}\lineto[r]^(.25){9}^(.75){7}&
  \overtag{\Circ}{-3}{6pt}\lineto[r]^(.25){5}^(.75){15}&
  \overtag{\Circ}{-2}{6pt}\lineto[ur]^(.25){3}^(.75){31}
  \lineto[dr]_(.25){2}_(.75){39}\\
  \overtag{\Circ}{-3}{6pt}\lineto[ur]_(.25){17}_(.75){3}
  &&&&\overtag{\Circ}{-2}{6pt}\\&}$$ and its splice diagram with leaf
weights is
$$\splicediag{15}{45}{
  \lefttag{\Circ}{Y_2}{6pt}\lineto[dr]^(.25){30}^(.75){2}
  &&&&\righttag{\Circ}{Y_3}{8pt}\\
  \hbox to 0pt{\hss$\Delta=$\qquad}
  &\Circ\lineto[rr]^(.25){9}^(.75){15}&&
  \Circ\lineto[ur]^(.25){3}^(.75){32}
  \lineto[dr]_(.25){2}_(.75){39}\\
  \lefttag{\Circ}{Y_1}{6pt}\lineto[ur]_(.25){17}_(.75){3}
  &&&&\righttag{\Circ}{Y_4}{8pt}\\&}$$ We have also shown the
association of $\C^4$--coordinates $Y_1,\dots,Y_4$ with leaves of
$\Delta$.

The discriminant group $D$ is
cyclic of order $33$. Its action on $\C^4$ can be read off from the
splice diagram and leaf weights, and is generated by the four
diagonal matrices corresponding to the four leaves ($\zeta=e^{-2\pi\,i/33}$):
\begin{align*}
e_1&=\langle \zeta^{17}, \zeta^{9},\zeta^{4},\zeta^{6}\rangle\\
e_2&=\langle \zeta^{9}, \zeta^{30},\zeta^{6},\zeta^{9}\rangle\\
e_3&=\langle \zeta^{4}, \zeta^{6},\zeta^{32},\zeta^{15}\rangle\\
e_4&=\langle \zeta^{6}, \zeta^{9},\zeta^{15},\zeta^{39}\rangle\,.
\end{align*}
In this case $e_1$ clearly suffices to generate the group.

Calling the left node $v$, the $v$--weights of the variables
$Y_1,Y_2,Y_3,Y_4$ are read off from the splice diagram as $18$,
$27$, $12$, and $18$. The GCD is $3$, so the reduced weights are
$6$, $9$, $4$, $6$. The total reduced weight for that node is
$54/3=18$, so admissible monomials for the two edges departing $v$
to the left must be $Y_1^3$ and $Y_2^2$. For the edge going right
there are two monomials of the desired weight: $Y_3^3Y_4$ and
$Y_4^3$. One checks that all these monomials transform with the
same character under the $D$--action ($e_1$ acts on each by
$\zeta^{18}$), so we can choose either one of $Y_3^3Y_4$ and
$Y_4^3$; we choose $Y_4^3$.  The number of equations to write down
for this node is $\delta_v-2=1$. We choose ``general
coefficients'' $1,1,1$ and write down
$$Y_1^3+Y_2^2+Y_4^3=0\,.$$

For the right node $v'$ the reduced $v'$--weights of $Y_1$, $Y_2$,
$Y_3$, $Y_4$ are $4$, $6$, $10$, $15$, and total reduced weight is $30$.
Hence admissible monomials are $Y_3^3$ and $Y_4^2$ for the edges going
right, and a choice of $Y_2^5$, $Y_1^3Y_2^3$, or $Y_1^6Y_2$ for the
leftward edge; again all choices are $D$--equivariant ($e_1$ acts by
$\zeta^{12}$ on all). We make our choices and write a second equation
$$Y_2^5+Y_3^3+Y_4^2=0\,.$$
The results of \cite{nw1} tell us that the two equations define a
normal complete intersection singularity $(V,0)$, that the action of
$D=\Z/33$ on it is free off the singular point, and the quotient is a
normal singularity $(X,o)$ with resolution graph $\Gamma$. A mental
calculation shows that our coefficients are in fact general; any other
choice can be reduced to these by diagonal coordinate transformation
of $\C^4$.

The End Curve Theorem tells us that if a singularity has this
resolution graph and has end curve functions for its four leaves, then
it is a higher weight deformation of the above example, i.e.,
(possibly) deformed by adding higher weight terms equivariantly in the
two equations.

\section{Overview of the proof}\label{sec:overview}

The End Curve Theorem was proved in \cite{nw2} when the link
$\Sigma$ is a $\Z$HS (so there is no group action). We first
outline the proof in this case. We thus assume we have an end
curve function $z_i$ on $X$ associated with each leaf of $\Gamma$.
By replacing $z_i$ by a suitable root if necessary one can assume
its zero-set is not only irreducible but also reduced (this uses
that $\Sigma$ is a $\Z$HS--cf. Section \ref{sec:proof}). The claim
then is that these functions $z_i$ generate the maximal ideal of
the local ring of $X$ at $o$ and that $X$ is a complete
intersection given by splice equations in these generators. The
main step is to show that each curve $C_i=\{z_i=0\}\subset X$ is a
complete intersection curve.

\subsection{}\label{step1} 
Consider the curve $C_1$ given by $z_1=0$.  For $j> 1$ the function
$z_j$, restricted to $C_1$, has degree of vanishing
$\ell_{1j}$ at $o$, so the numbers $\ell_{1j}$, $j=2,\dots,t$,
generate a numerical semigroup $S\subset \N$ which is a sub-semigroup
of the full value semigroup $V(C_1)$ (the semigroup generated by all
degrees of vanishing at $o$ of functions on $C_1$). The $\delta$--invariant
$\delta(S)$ and the $\delta$--invariant $\delta(C_1):=\delta(V(C_1))$
therefore satisfy
\begin{equation}
  \label{eq:s1}
2\delta(C_1)\le 2\delta(S)
\end{equation}
(the $\delta$-invariant counts the number of gaps in the semigroup,
i.e., the size of $\N-S$).

\subsection{}\label{step2} Classical theory of numerical semigroups,
developed further in Theorem 3.1 of \cite{nw2},  shows via an induction over
subdiagrams of $\Gamma$ that
\begin{equation}
  \label{eq:s2}
2\delta(S)\le 1+\sum_{v\ne1}(\delta_v-2)\ell_{1v}\,,
\end{equation}
with equality if and only if the semigroup condition holds for
$\Delta$ at every vertex and edge pointing away from $1$. Moreover, if
equality holds, the monomial curve for this semigroup is a complete
intersection, with maximal ideal generated by $z_2,\dots,z_n$.

\subsection{}\label{step3}
The $\delta$-invariant of a curve is determined by Milnor's $\mu$
invariant, which in our case is a topological invariant, computable
in terms of the splice diagram (Sect.\ 11
of \cite{eisenbud-neumann}, see also Lemma \ref{le:chi of fiber}) as
\begin{equation}
  \label{eq:s3}
2\delta(C_1)=1+\sum_{v\ne1}(\delta_v-2)\ell_{1v}\,.
\end{equation}

\subsection{}\label{step4}
Comparing \eqref{eq:s1}, \eqref{eq:s2}, and \eqref{eq:s3}, we must
have equality in \eqref{eq:s1} and \eqref{eq:s2}. It follows that
$S=V(C_1)$. Moreover, by step \ref{step2}, the monomial curve for
$S$ is a complete intersection with maximal ideal generated by
$z_2,\dots,z_t$. Since $C_1$ is a positive weight deformation of
this monomial curve, $C_1$ is also a complete intersection,  with
maximal ideal generated by $z_2,\dots,z_t$. It follows that
$(X,o)$ is a complete intersection with maximal ideal generated by
$z_1,\dots,z_t$.

\subsection{}\label{step5}
Repeating the above for all leaves $i=1,\dots,t$ shows that the
semigroup conditions hold. One can thus choose admissible monomials
for every node, and it is then not hard to deduce that equations of
splice type hold. Finally, one deduces that $(X,o)$ is defined by these
equations, completing the proof.

\subsection{General case}
We must now describe how the above proof is modified when $\Sigma$ is
not a $\Z$HS, so $D$ is non-trivial. It is not hard to show that the
functions $x_i$ (appropriate roots of the end curve functions $z_i$)
are defined on the \UAC{} $V$ of $X$. But the curve
$C_1=\{x_1=0\}$ is no longer an irreducible curve, so the theory of
numerical semigroups of \ref{step2} above cannot be used.  \comment{[WDN: Do we
need the following? It seems to me to be implicit] The need to
construct an alternate set-up is what accounts for the time lag
between \cite{nw1} and the current paper.  In particular, Sections 3
through 7 are needed to prove Theorem \ref{th:maj}, whose proof in the
case of trivial $D$ is short.}

We extend the theory of value semigroups and the appropriate results
concerning them to reducible curves that have an action of a group $D$
that is transitive on components; the value semigroup is now a
subsemigroup of $\N\times \hat D$, where $\hat D$ is the character
group of $D$. The inductive argument of step \ref{step2} must now deal
with subsemigroups of a semigroup $\N\times \hat D$, where $\hat D$
changes at each step of the induction. Moreover, we must show in the
end that this extension allows us to deduce, as before, that $V$ is a
complete intersection with maximal ideal generated by the $x_i$, that
the semigroup conditions hold, which guarantee that admissible
monomials exist, but also that the congruence conditions hold,
allowing us to choose the monomials $D$--equivariantly.  Once this is
done, one again deduces that equations of splice type hold on
$V$. Finally, using the main theorem of \cite{nw1} one deduces that
$V$ is defined by these equations and that $X=V/D$, thus completing
the proof.

The necessary theory of reducible curves and their value semigroups is
developed in sections \ref{sec:D-curves}--\ref{sec:nu} and the proof is
completed in section \ref{sec:proof}. Some needed topological computations
are collected together in section \ref{sec:top}.

\section{$D$--curves}\label{sec:D-curves}
A \emph{$D$--curve} is a reduced curve germ $(C,o)$ on which a finite
abelian group $D$ acts effectively (i.e., $D\rightarrow \text{Aut}(C)$
is injective) and transitively on the set of branches.  Denote the
branches $(C_{i},o)$, $i=1,\dots,r$.  If $H\subset D$ is the subgroup
stabilizing (any) one branch, then $F:= D/H$ acts simply transitively
on the set of branches of $C$ (recall that any effective transitive
action of an abelian group is simply transitive).  $D$ also acts on
the normalization $\tilde{C}$ of $C$, a disjoint union of $r$ smooth
curves $\tilde{C_{i}}$.  On the level of analytic local rings, $D$
acts on $(R,\m)$ (the local ring of $C$), on the direct sum of its
branches $R/\mathcal P _{i}$, and on its normalization $\tilde{R}=
\oplus_{i=1}^r \tilde{R_{i}}$, where each $\tilde{R_{i}}$ is a
convergent power series ring $\C\{\{y_{i}\}\}$.  One may assume that
the parameters $y_i $ form one $D$--orbit (up to multiplication by
scalars). $H$ then acts on each $y_i $ via the same character,
independent of $i$.\comment{WDN: I've rewritten and shortened this
  section and the next from this point on to give a definition of the
  value semigroup and organize better}

The natural valuation $v_{i}$ on $\tilde{R_{i}}$ induces one on
$\tilde{R}$ by value on the $\tilde{R_{i}}$--component; the induced
valuation on $R$ is given by order of vanishing of a function $f\in R$
along the branch $C_{i}$.  (Of course, define $v_{i}(0)=\infty$.)  $D$
permutes the branches and hence the valuations, with
$$v_{i}(\sigma (f))=v_{\sigma (i)}(f) \text{ for all }  \sigma \in D.$$
Thus, if $f$ is an eigenfunction for the $D$--action, then
$v_i(f)=v_j(f)$ for all $i,j$; we then just write $v(f)$.
\begin{definition}\label{def:value semigroup} 
  Denote the character group of $D$ by $\hat D$.  The \emph{value
    semigroup} $\calS(C)$ of the $D$--curve $(C,o)$ is the
  subsemigroup of $\N\times\hat D$ consisting of all pairs
  $(v(f),\chi)$ with $\chi\in \hat D$ and $f\in R$ a
  $D$--eigenfunction with character $\chi$.
\end{definition}

$\tilde{R}$, and hence $R$, has a natural $D$--filtration given
by the ideals
$$J_{n}=\{f~|~v_{i}(f)\geq n \text{ for all } i\}.$$ We denote by
$\tilde R'$ and $R'$ the associated gradeds for this filtration.  $R'$
is the graded ring of a reduced weighted homogeneous curve $C'$ (thus
a union of ``monomial curves''), again with an effective action of $D$
acting transitively on the $r$ branches. Each $J_{n}\tilde
R/J_{n+1}\tilde R$ is for $n\geq 0$ a vector space of dimension $r$
and the associated graded ring $\tilde R'=Gr_{J}\tilde{R}$ is a direct
sum of $r$ polynomial rings in one variable (of degree 1).

Recall that the \emph{delta invariant} $\delta(C)=\delta (R)$ is the
length of the $R$--module $\tilde{R}/R$, which in this case is its
dimension as a $\C$--vector space.

\begin{proposition}\label{prop:delta} 
  $\delta(R)=\delta(R')$ and $\calS(C)=\calS(C')$.  The
  value semigroup $\calS(\tilde C)$ of the normalization has exactly $r$ elements in
  each degree. The delta-invariant of $R$ (or $C$) is equal to the
  size of the complement of $\calS(C)$ in $\calS(\tilde C)$.
\end{proposition}
\begin{proof}
  The delta invariant $\delta (R)$ is finite, so for $n$ sufficiently
  large, $J_{n}R=J_{n}\tilde R$.  Computing
  $\delta(R)$ by summing over graded pieces of $R$ and $\tilde{R}$,
  one sees that $\delta(R)=\delta(R')$  (this is a general fact about
  reduced curves).

$\tilde R$ splits (as a vector space) as a sum $\tilde R=\bigoplus
\tilde R_\chi$ over the characters of $D$ and this splitting is
compatible with the filtration $\{J_n\}$ (in the sense that
$J_n\,\tilde R_\chi=(J_n\tilde R)_\chi$ for any $D$--character
$\chi$). Thus $\calS(\tilde R)=\calS(\tilde R')$. The same
argument applies to show $\calS(R)=\calS(R')$.
  
The following lemma shows that the splitting by characters splits each
$J_n\tilde R/J_{n+1}\tilde R$ into $r$ $1$--dimensional summands (and
trivial summands for the remaining $|D|-r$ characters), and hence also
splits the subspaces $J_nR/J_{n+1}R$ into $1$--dimensional
summands. Since we can compute the delta invariant $\delta(R')$ by
counting these summands, the proposition then follows.
\end{proof}
Consider the natural
exact sequence of character groups
$$0\rightarrow \hat{F}\rightarrow \hat{D}\rightarrow
\hat{H}\rightarrow 0.$$
\begin{lemma}\label{le:char} The $D$--eigenfunctions of
  $J_{n}\tilde R/J_{n+1}\tilde R$ ($n>0$) all have the same
  $H$--character.  A collection of them are $\C$--linearly independent
  if and only if their $D$--characters are distinct.  They form a basis
  of $J_{n}\tilde R/J_{n+1}\tilde R$ if and only if their
  $D$--characters form exactly one $\hat{F}$--coset of $\hat{D}$.
\end{lemma}
\begin{proof} Diagonalize the action of $D$ on the $r$--dimensional
  space $J_{1}\tilde R/J_{2}\tilde R$.  As mentioned above, $H$ acts
  via one character for this space, so the $D$--characters involved
  form one fiber of the map $\hat{D}\rightarrow\hat{H}.$ Choose any
  $D$--eigenfunction $g\in J_{1}/J_{2}$.  Then multiplication by
  $g^{n-1}$ maps $J_{1}\tilde R/J_{2}\tilde R$ isomorphically onto
  $J_{n}\tilde R/J_{n+1}\tilde R$,
  shifting the characters by $(n-1)$ times the character of $g$.  So,
  $r$ distinct $D$--characters appear in $J_{n}\tilde R/J_{n+1}\tilde R$.
\end{proof}
\section{Weighted homogeneous $D$--curves}
\label{sec:more}

Continue the setup of the last section, but assume at the start that
$R=R'$ is a positively graded ring, so that $C$ is a weighted
homogeneous curve with a $D$--action. We can choose a (not
necessarily minimal) set of homogeneous generators $x_{i}$,
$i=1,\dots,m$, of $R$, of weights $\ell_{i}$ (necessarily with greatest
common divisor 1), so that $D$ acts on $x_{i}$ via a character
$\chi_{i}$.  The value semigroup $\calS(C)$ is then generated by
the elements $(\ell_i,\chi_i)$.

The $x_{i}$ embed $C\subset \C^{m}$. By scaling the $x_i$ if necessary
we may assume $(1,1,\dots,1)\in C$.  Then one of the $r$ branches of
$C$ is the monomial curve 
$\{(u^{\ell_{1}},\dots,u^{\ell_{m}}):u\in \C\}$.

There are two subgroups of $(\C^*)^{m}$ which act on $C$ via
multiplication in the entries: $D$, embedded via
$(\chi_{1},\dots,\chi_{m})$; and $\C^*$, via $u\mapsto
(u^{\ell_{1}},\dots,u^{\ell_{m}})$.

\begin{definition} $G\subset (\C^*)^m$ is the subgroup generated by $D$
  and $\C^* $, embedded as above.
\end{definition}

$D$ acts transitively on the branches of $C$, so $G$ acts simply
transitively on $C-\{0\}$.  Thus, we may view $G=C-\{0\}$, while $C$
is the closure of $G$ in $\C^m$.

Consider next $H:= D\cap\C^*$, the subgroup of $D$ which stabilizes
the branch given by the monomial curve through $(1,1,\dots,1)$.  As a
subgroup of $\C^*$, it is cyclic, and the quotient group $F:=
D/H$ has order $r$, the number of branches of $C$.  Using the inclusion
map of $H$ to $D$, and the negative of its inclusion map to $\C^*$, we
deduce exact sequences:
\begin{equation}
  \label{eq:exact1}
  1\rightarrow H\rightarrow \C^* \times D\rightarrow
  G\rightarrow 1  \,,
\end{equation}
\begin{equation}
  \label{eq:exact2}
  1\rightarrow \C^*\rightarrow G\rightarrow F\rightarrow
  1\,.
\end{equation}

Since $\C^*$ is a divisible group, the exact sequence
\eqref{eq:exact2} splits (but not canonically).  This means that there
exists a subgroup of $G$ (though not necessarily of $D$), isomorphic
to $F$, which acts on the curve and acts simply transitively on the
set of branches.  Dualizing \eqref{eq:exact2}, we get a natural exact
sequence of groups of characters
\begin{equation}
  \label{eq:exact3}
0\rightarrow \hat{F}\rightarrow \hat{G}\rightarrow \Z
   \rightarrow 0\,,
\end{equation}
which splits non-canonically to give an abstract isomorphism
$$\hat{G}\cong \Z\oplus \hat{F}\,.$$
Dualizing \eqref{eq:exact1} gives a natural inclusion
$$\hat{G}\subset \Z \oplus \hat{D}\,.$$

The inclusion $G\subset (\C^*)^m$ induces a surjection $\Z^m\to \hat
G$, giving $m$ natural characters $\bar\chi_{k}\in \hat G$,
$k=1,\dots,m$, which generate $\hat G$.  The image of $\bar\chi_{k}$
in $\Z \oplus \hat{D}$ is $(\ell_{k},\chi_{k})$.  So any character of
$G$ may be written additively as\footnote{We write $\hat G$ additively
  in the following even though it is more natural to think of $\hat D$
  as a multiplicative group.}
$$\hat I:=\sum_{k=1}^{m} i_k\bar\chi_k,$$
where $I=(i_1,\dots,i_m)\in \Z^m$, and $I$ and $I'$ represent equal
characters $\hat I=\hat I'\in\hat G$ if and only if their weights and 
 $D$--characters are equal.
\begin{definition} A character $X\in \hat{G}$ is called
  \emph{non-negative} if it is a linear combination of the
  $\bar\chi_k$ with non-negative coefficients.
Thus, the set of non-negative characters forms the value
semigroup $\calS(C)$ as a subsemigroup of $\hat G\subset \Z\times\hat
D$. The value semigroup $\calS(\hat C)$ of the normalization is
$\pi^{-1}(\N)\subset \hat D$. So $\delta(C)$ is the number of elements
of $\pi^{-1}(\N)$ that are not non-negative characters.
\end{definition}

The $x_{i}$ embed $C\subset \C^{m}$. 
Let $I=(i_1,\dots,i_m)$ be an $m$--tuple, with all $i_k\geq 0$.  Each
monomial $x^{I}:= x_{1}^{i_{1}}x_{2}^{i_{2}}\dots x_{m}^{i_{m}}$ is
homogeneous of weight $\sum i_k \text{wt}(x_k)$ and transforms
via 
the character $\chi_I:= \prod \chi_k^{i_k} \in \hat{D}$.  Introducing weighted
coordinates $Y_1 ,\dots,Y_{m}$ in $\C^{m} $, we can summarise
results of this and the previous section.
\begin{proposition}\label{prop:G} $x^{I}$ equals a constant times
  $x^{I'}$ if and only if their weights and $D$--characters are
  equal.  With the $x_{i}$ scaled so that
  $(1,1,\dots,1)\in C'\subset\C^{m}$,  the $D$--curve $C\subset\C^{m}$ is
  the closure of the $G$--orbit of $(1,1,\dots,1)$. The ideal of $C\subset
  \C^{m} $ is generated by all differences $Y^{I}-Y^{I'}$ of monomials
  with equal $G$--characters: $\hat I=\hat I'$.

  The finite abelian ``group of components'' $F=G/\C^*$ is dual (hence
  isomorphic) to the torsion subgroup of $\hat{G}$.

  The delta-invariant $\delta(R)=\delta(R')$ is computed by summing,
  over all $n\geq 0$, $r$ minus the number of monomial elements $x^I $
  of weight $n$ and distinct $D$--characters.\qed
\end{proposition}
We will need the following proposition.
\begin{proposition}\label{prop:M} The inclusion map $D\subset G$ gives
  a surjection $\hat{G}\twoheadrightarrow \hat{D}$, whose kernel is the
  cyclic group generated by a distinguished character $\hat Q$, whose
  weight is the order of $H$, i.e. $|D|/|F|$.  Its image in $\Z
  \oplus \hat{D}$ is $(|H|,0)$.
\end{proposition}
\proof This is immediate from the sequence of kernels
  of the  surjection of short exact sequences:
$$\xymatrix{
0\ar[r]& \hat G\ar[r]\ar@{->>}[d] &\Z\oplus \hat D\ar[r]\ar@{->>}[d] &\hat
H\ar[r]\ar@{->>}[d]&0\\
0\ar[r]&\hat D\ar[r]&\hat D\ar[r]&0\ar[r]&0\hbox to 0pt{\hskip 100pt\qed\hss}}$$

\medskip
This proposition can be restated as follows: Suppose $\hat{I}$
represents a character of $G$ for which the image in $\mathbb Z \oplus
\hat{D}$ is of the form $(\ell',0)$; then $\ell'$ is a multiple of
$|D|/|F|$, and more precisely $\hat{I}$ is a multiple of
$\hat Q$.  It will be of special interest to us when $\hat Q$ is a
non-negative character, i.e., $\hat Q\in\calS$. In this case we can
represent $\hat Q$ by a non-negative tuple $Q$, and the monomial $Y^Q$
will play a special r\^ole for us.

\medskip The following Corollary of Proposition
\ref{prop:G} clarifies the relationship
between the value semigroup and weighted homogeneous $D$--curve.
\begin{corollary}\label{cor:classification}
  A weighted homogeneous $D$--curve $C$ is determined, up to
  isomorphism of weighted homogeneous curves, by its value semigroup
  $\calS(C)$.  Its graded ring $R$ is the semigroup ring
  $\C[\calS(C)]$. The $D$--action is determined by the surjection
  $\calS\twoheadrightarrow\hat D$.

  A commutative semigroup $\calS$ and surjection
  $\chi\colon\calS\twoheadrightarrow\hat D$ determines a weighted
  homogeneous $D$--curve if and only if $\calS$ satisfies the
  cancellation law and is not
  a group, and the induced homomorphism of its group of quotients to
  $\hat D$ has infinite cyclic kernel.
\end{corollary}
\begin{proof} The first paragraph is just a reinterpretation of the
  first part of Proposition \ref{prop:G}: The isomorphism
  $\C[\calS(C)]\to R$ is given by $\hat I\mapsto Y^I$. The grading is
  determined by the map $\calS(C)\to \N$ given by mapping $\calS(C)$
  into its group of quotients $\hat G$ and then factoring by the
  torsion subgroup of $\hat G$.  $D$ acts on elements
  $X\in\calS(C)$---i.e., generators of $\C[\calS(C)]$---by
  $g\cdot X=\chi(X)(g)X$.

We won't use the
  second part of the corollary so we leave it to the reader.
\end{proof}

\section{An instructive example}\label{sec:ex}

Let $n_1,\dots,n_m\geq 2$ be integers, and consider the curve
$C\subset \C^m$ defined by
$$Y_1^{n_1}=Y_i^{n_i}\,,\quad i=2,\dots,m.$$
Write \ $\calN=n_1\dots n_m$, \ $\calN_i=\calN/n_i$, \
$s=\operatorname{GCD}(\calN_1,\dots,\calN_m)$.

\begin{example}\label{ex:instr} With the above notation,
\begin{enumerate}
\item $C$ is a reduced weighted homogeneous curve with (relatively
  prime) weights $\calN_1/s,\dots, \calN_m/s$; a particular branch
  $C_1$ is the irreducible monomial curve
  $(u^{\calN_1/s},\dots,u^{\calN_m/s}).$
\item Let $G:= C-\{0\}\subset (\C^*)^m$; $G$ is a subgroup of
  $(\C^*)^m$, acting on $C$ by coordinate-wise multiplication, and
  simply transitively on $C-\{0\}$.
\item  The connected component of the identity of $G$ is
 $C_1-\{0\}$, a copy of $\C^*$.
\item $G/\C^*$ is a finite abelian group whose order $r$ (the number
  of components of $C$) equals $s$ (the GCD above).
\item There is a (non-canonical) splitting $G=\C^*\times D'$, so that
  $C$ is a $D'$--curve, and $D'$ acts simply transitively on the set of
  branches.
\item The Milnor number $\mu$ and delta-invariant $\delta$ of $C$
  satisfy
$$\mu -1=2\delta -r=(m-1)\calN\ -\sum_{i=1}^{m}\calN_{i}.$$
\end{enumerate}
\end{example}
\begin{proof} Statements (1)--(3) and (5) are obvious.  We show (4).
  For any $a,a'\neq 0$, one has that the cardinality $|C\cap\{Y_1
  =a\}|=n_2\dots n_m =\calN_1$, while $|C_1 \cap\{Y_1 =a'\}|$ is the
  weight $\calN_1 /s$ of $Y_1$.  Since $G$ acts transitively on
  $C-\{0\}$, it follows that any branch intersects $\{Y_1=a'\} $ in
  $\calN_1/s$ points, so that there must be $s$ branches.

  An alternative proof of (4) that does slightly more is as
  follows. Since $\C^*$ is a divisible group, one has a splitting
  $G=\C^*\times D'$, where $D'\subset G$ is a finite subgroup mapping
  isomorphically onto $G/\C^*$.  Then $C$ is a $D'$--curve, and $D'$
  acts simply transitively on the set of branches.  As in Proposition
  (\ref{prop:G}), $D'$ is isomorphic to the torsion subgroup of the
  group with generators $e_1,\dots,e_m$, and relations $n_1e_1=n_ie_i
  , i=2,\dots, m.$  This implies again that $s$ is the number of
  branches, and also allows calculation of the elementary divisors of
  $D'$ via the ideals of minors of the matrix associated to the
  relations.

  The familiar formula \cite{buchweitz-greuel} relating Milnor number
  and delta invariant ($\mu =2\delta-r+1$), plus the calculation of
  $\mu$ in e.g., \cite{gh} (or see Lemma \ref{le:chi of fiber}),
  yields the last assertion.
 \end{proof}

\section{$D$--curves from rooted resolution diagrams}
\label{sec:curves from diagrams}

Let $\Gamma$ be a resolution diagram which is a tree, with
distinguished leaf $*$, called the \emph{root}; we say $(\Gamma, *)$
is a \emph{rooted tree}.  These data will give rise to $C(\Gamma,*)$,
a reduced and weighted homogeneous $D$--curve, where $D := D(\Gamma)$
is the discriminant group associated to $\Gamma$ (i.e., the cokernel
of the intersection matrix of $\Gamma$).  We derive key properties of
$C$ by its inductive relationship with subtrees (with fewer nodes).

For the moment, assume $\Gamma$ is not a string, i.e., has at least
one vertex of valency $\geq 3$\comment{WDN: if \emph{really} isn't
  necessary, but you wanted this case treated separately.} (this is
not strictly necessary).  Order the non-root leaves
$w_{1},\dots,w_{m}$, and index them by variables $Y_{1},\dots,Y_{m}$.
The $m$ linking numbers (Definition \ref{def:linking}) $\ell_{i}:=
\ell_{w_{i},*}$ from the non-root leaves to $*$ give a set of
(non-reduced) weights for the variables $Y_i $; one gets reduced
weights by dividing by
\begin{equation}
  \label{eq:s}
  s:=\operatorname{GCD}(\ell_1,\dots,\ell_m).
\end{equation}
These weights give a copy of $\C^*$ in the diagonal group
$(\C^*)^{m}$.  (The linking numbers, and hence $s$, can be read off
from the splice diagram $\Delta$ associated to $\Gamma$.)

As described in subsection \ref{subsec:action},
every vertex $v$ of $\Gamma$ gives an element $e_{v}\in D$, and $D$ is
generated by the elements $e_{w_{i}}$, $i=1,\dots,m$.  Moreover, $D$
has a natural non-degenerate $\Q /\Z$--valued bilinear form
$(e,e')\mapsto e\cdot e'$, induced by the intersection pairing for
$\Gamma$, and
one has an embedding of $D$ into
$(\C^*)^m$, where the image of $e\in D$ is the $m$--tuple whose $j$th
entry is the root of unity $e^{2\pi i \,e\cdot e_{w_{j}}}$.
Recall
from subsection \ref{subsec:action} that for vertices $w,w'$ of
$\Gamma$, one has
\begin{equation}
  \label{eq:1}
 e_{w} \cdot e_{w'}=-\ell_{ww'}/ |D|\,,
\end{equation}
which for nodes and leaves can be read off from $|D|$ and the splice diagram
$\Delta$ with leaf weights.

As before, define $G\subset (\C^*)^{m}$ to be the subgroup generated
by $D$ and the $\C^*$.

\begin{definition}\label{def:61} The $D$--curve $C:= C(\Gamma,*)$ associated to
  the rooted diagram $(\Gamma,*)$ is the closure of $G$ in $\C^m$, or
  equivalently the closure in $\C^m$ of the $G$--orbit of
  $(1,1,\dots,1)$.
\end{definition}

\comment{WDN: reorganized these two paragraphs and moved equation (7)
  to before the Lemma.}%
From the natural map $\Z^m\rightarrow \hat{G}$, any element of
$\hat{G}$ may be represented as before by $\hat{I}=\sum
_{j=1}^{m}i_j\bar\chi_j$, where this
character gives the weight $\sum_{j=1}^{m}i_{j}\ell_{w_{j}}/s$, and
sends an element $e\in D$ to
$\prod_{j=1}^{m}e^{2\pi i\,i_{j}\,e\cdot e_{w_{j}}}$.  If
$I=(i_1,\dots,i_m)$ is non-negative $\hat I$ represents the weight and
$D$--character of the monomial $Y^I$.

The ideal of $C$ in $\C[Y_1,\dots,Y_m]$ is generated by $Y^I-Y^{I'}$
for all non-negative $m$--tuples $I$ and $I'$ with the same weight and
$D$--character (i.e., giving the same image in $\hat{G}$).  Denote by
$r$ the index in $G$ of the stabilizer $H=\C^*\cap D$ of one branch of
$C$.  Thus $r=|G/\C^*|$ is the number of branches of $C$.  We shall
prove later that $r=s$ ($s$ as in \eqref{eq:s}), hence also this
number can
be read off the splice diagram.

The above applies also when $\Gamma$ is a
string, corresponding to a cyclic quotient singularity of some order
$|D|$. Suppose there are two leaves, $*$ and $w$.  Then $D$ is the
cyclic group generated by $e_{w}$, acting on $\C$ by $e^{2\pi i\,
  e_w\cdot e_w}=e^{-2\pi i\ell_{ww}/|D|}$.
There is a single variable $Y$, with weight $\ell_{w,*}=1$; we have
$G=\C^*$, $D\subset G$ the cyclic subgroup of order $|D|$.  The
associated $D$--curve is simply a copy of $\C$ (and so the semigroup
is $\N$).  When there is only the one vertex $*$, it plays the role of
both leaf $w$ and root $*$, and all is the same, with $\ell_{w,*}=1$
and $D$ trivial.

\medskip
For general $(\Gamma,*)$ our goal is to do an inductive comparison
between the data of $(\Gamma,*)$ and data from rooted
sub-diagrams.  We start with the case of one node.

\begin{proposition}\label{pr:one} Consider a minimal resolution graph
  $\Gamma$ with one node,
given by the diagram
$$
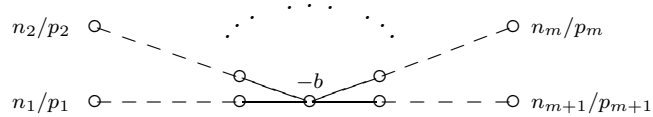
\xymatrix@R=4pt@C=24pt@M=0pt@W=0pt@H=0pt{\\
\lefttag{\Circ}{n_2/p_2}{8pt}\dashto[ddrr]&
&\hbox to 0pt{\hss\lower 4pt\hbox{.}.\,\raise3pt\hbox{.}\hss}
&\hbox to 0pt{\hss\raise15pt
\hbox{.}\,\,\raise15.7pt\hbox{.}\,\,\raise15pt\hbox{.}\hss}
&\hbox to 0pt{\hss\raise 3pt\hbox{.}\,.\lower4pt\hbox{.}\hss}
&&\righttag{\Circ}{n_{m}/p_{m}}{8pt}\dashto[ddll]\\
\\
&&\Circ\lineto[dr]&&\Circ\lineto[dl]\\
\lefttag{\Circ}{n_1/p_1}{8pt}\dashto[rr]&&
\Circ\lineto[r]&\overtag{\Circ}{-b}{8pt}\lineto[r]&\Circ
\dashto[rr]&&\righttag{\Circ}{n_{m+1}/p_{m+1}}{8pt}\\&~\\&~}
$$
Strings of $\Gamma$ are described by the continued fractions shown,
starting from the node.  Let $*$ denote the leaf on the lower right,
$C$ the associated $D$--curve, with discriminant group $D$ of order
$|D|=n_1\dots n_{m+1}(b-\sum_{i=1}^{m+1}p_{i}/n_{i})$.  Then $C$ and
the group $G$ are as in Example \ref{ex:instr}, and depend only on
$n_1,\dots,n_m$.  In particular, $r=s$, i.e., the number of branches
is the GCD of the linking numbers to $*$.
\end{proposition}
\begin{proof} It is shown in \cite{neumann83} that a weighted
  homogeneous surface singularity with resolution graph $\Gamma$ is a
  splice quotient, with \UAC{} defined by
     \begin{gather*}
Y_1^{n_1}-Y_2^{n_2}+a_2Y_{m+1}^{n_{m+1}}=0\\
...\qquad...\qquad...\\
     Y_1^{n_1}-Y_m^{n_m}+a_mY_{m+1}^{n_{m+1}}=0
\end{gather*}
Here, the $a_i$ are distinct and non-zero, giving the analytic type of
the $m+1$ intersection points on the central curve.  $D(\Gamma)$ acts
on the coordinates $Y_i$ according to the usual characters, and acts
freely off the origin.  The curve cut out by $Y_{m+1}=0$ is the one in
Example \ref{ex:instr}; its quotient by $D$ is irreducible (see also
\cite[Theorem 7.2(6)]{nw1}), so $D$ acts transitively on the branches.
By definition, the curve $C(\Gamma,*)$ in $\C^m$ consists of the
monomial curve with weights $\calN _i/s$ (which is a component of
$\{Y_{m+1}=0\}$) and its translates by $D$.  It follows that
$C(\Gamma,*)$ is equal to $\{Y_{m+1}=0\}$.
\end{proof}

\begin{remark} Note that in this example, the rooted splice
  diagram

$$
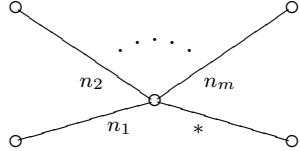
\splicediag{10}{36}{
  \Circ&&\Circ\\ \\
  &\vbox to 0 pt{\vss\hbox{\raise6pt\hbox{.}\;\raise9pt\hbox{.}%
  \;\raise10pt\hbox{.}\;\raise9pt\hbox{.}\;\raise6pt\hbox{.}\hss}}
\\
  &\Circ\lineto[dl]^(.3){n_1}\lineto[uuul]^(.35){n_2}
  \lineto[uuur]_(.35){n_m}\lineto[dr]_(.35){*}\\
  \Circ&&\Circ\\~}$$
uniquely determines both the curve $C$ and the group
$G$, but \emph{not} the group $D$.\comment{JMW--couldn't the diagram
  look prettier? WDN is it better now?}
\end{remark}

We return to the inductive comparison of the data of $(\Gamma,*)$ and
data from rooted sub-diagrams. Before starting, if any two
nodes in $\Gamma$ are adjacent, we blow up $\Gamma$ once in between
these,  so that $\Gamma$ has at least one vertex between any two adjacent
nodes (cf.\ Definition 6.1 of \cite{nw1}). Suppose $*$ is the leftmost
leaf in the resolution diagram\comment{JMW--would a picture
  help, as in the Appendix to \cite{nw1}? WDN This one OK?}
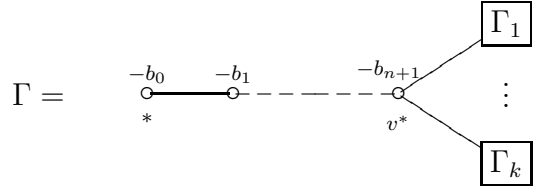
$$\splicediag{8}{30}{
  &&&&&\frame{\Gamma_1}\\
  \Gamma=&\undertag{\overtag\Circ{-b_0}{8pt}}{*}{4pt}
  \lineto[r]&\overtag\Circ{-b_1}{8pt}\dashto[r]&
\dashto[r]&\undertag{\overtag\Circ{-b_{n+1}\;\;\;}{10pt}}{v^*}{5pt}
  \lineto[ur]\lineto[dr]
  &\Vdots\\
  &&&&&\frame{\Gamma_k} \\~
}$$
 So, remove from $\Gamma$
the leaf $*$, the adjacent node $v^{*}$, and all the vertices in
between. This produces $k$ new rooted resolution diagrams
$(\Gamma_{i},*_{i})$.
The new $*_{i}$ is the vertex of $\Gamma_{i}$ closest to $v^{*}$, and it is
now a leaf.  We shall use notations like $G_{i}$, $m_{i}$, $D_{i}$,
$\Delta_{i}$ (associated splice diagrams), $r_i$ (number of branches),
$s_i$ (GCD of weights), etc., when referring to the new rooted
diagrams.  Note that the total number of non-root leaves is
$m_{1}+\ldots +m_{k}=m$.

In what follows, ``leaf'' shall mean non-root leaf.  As before, we
use notation $\ell_{ww'}$ and $\ell_{w}:= \ell_{w,*}$ for linking
numbers in $\Delta$ (as a splice diagram with leaf weights,
so $\ell_{ww'}$ is also defined if $w=w'$ is a leaf).  When computing
linking numbers in $\Delta_{i}$, we use notation
$\tilde{\ell}_{ww'}$ and $\tilde{\ell_{w}}:=\tilde{\ell}_{w,*_{1}}$.
The weights around the distinguished node $v^{*}$
are $|D_{1}|,\dots,|D_{k}|$ (by definition of the splice diagram) and
$c$, the weight in the direction of $*$.

\begin{lemma}\label{le:link} We have the following relations:
  \begin{enumerate}
  \item\label{linkit:1} For $w$ a leaf in $\Delta_{1}$,
    $$\ell_{w}=\tilde{\ell}_{w}\cdot |D_{2}|\dots
    |D_{k}|=\tilde{\ell_{w}}\cdot \ell_{v^*v^*}/c|D_{1}|.$$
  \item\label{linkit:2} For $w$ a leaf in $\Delta_{1}$, $w'$ a leaf in
    $\Delta_{i}$, $i>1$,
    $$\ell_{v^*v^*}\ell_{ww'}=c^{2}\ell_{w}\ell_{w'}.$$
  \item\label{linkit:3} For $w$, $w'$ leaves in $\Delta_{1}$,
    $$\ell_{ww'}/|D| =\tilde{\ell}_{ww'}/|D_{1}| +
    \tilde{\ell}_{w}\tilde{\ell}_{w'}\ell_{v^*v^*}/|D| |D_{1}|^{2}.$$
  \end{enumerate}
\end{lemma}

\begin{proof}
  The first two claims are straightforward. For the third, we first
  need some notation. For any two vertices $v$ and $v'$, recall
  $\ell_{vv'}$ is the product of splice diagram weights adjacent to
  the path from $v$ to $v'$; we will denote by $\ell'_{vv'}$ the product
  of splice diagram weights adjacent to the path from $v$ to $v'$, but
  excluding weights at $v$ and $v'$ (so $\ell'_{vv'}=1$ if $v=v'$).\comment{JMW--suppose $v=v'$? WDN
    an empty product evaluates to 1}

Let now $v$ be the vertex of $\Gamma$ that is closest to $v^*$ on the
path from $w$ to $w'$ (this is $w$ itself if $w=w'$). Rewrite the
equation to be proved first as
 $$\ell_{ww'}|D_1| -\tilde{\ell}_{ww'}|D| =
 \tilde{\ell}_{w}\tilde{\ell}_{w'}\ell_{v^*v^*}/|D_{1}|\,.$$ If $a$,
 $d_w$, $d_{w'}$ are the weights in $\Delta$ at $v$ towards $v^*$,
 $w$, $w'$ respectively, $R$ the product of the remaining weights
 at $v$, and $\tilde a$ the weight in $\Delta_1$ at $v$ towards
 $v^*$, then we can rewrite our equation
 $$\ell'_{wv}\ell'_{w'v}aR|D_1| -\ell'_{wv}\ell'_{w'v}\tilde aR|D| =
 (\ell'_{wv}d_{w'}R\ell'_{vv*})(\ell'_{w'v}d_{w}R\ell'_{vv*})
 \ell_{v^*v^*}/|D_{1}|\,.$$
Cancelling common factors from this equation simplifies it to
$$a|D_1|-\tilde a|D|=
d_wd_{w'}c(\ell'_{vv^*})^2\ell_{v^*v^*}/|D_1|\,.$$
This equation is  what was proved in Lemma
12.7 of \cite{nw1} (there $v$ was a node, but that was not necessary
to the proof)\comment{JMW--make sure I check this! WDN You can trust me.}.
\end{proof}

We state the key results needed to do induction.

\begin{theorem}\label{th:compat} Enumerate the $m$ non-root leaves
  of $\Gamma$ so that the first $m_{1}$ of them yield exactly the
  non-root leaves of $\Gamma_{1}$.  Compose the inclusion $G\subset
  (\C^*)^{m}$ with the projection onto $(\C^*)^{m_{1}}$ given by the
  first $m_{1}$ entries.  Then the image is exactly $G_{1}$.
\end{theorem}

\begin{proof}
  We show the image of the projection map $\pi:G\subset (\C^*)^{m}
  \rightarrow (\C^*)^{m_{1}}$ is exactly $G_{1}$.  Now, an inclusion
  $\C ^{*}\subset \C ^{*N}$ is described by a projective $N$--tuple of
  rational numbers:
  $$[\alpha_{1}:\alpha_2:\dots:\alpha_{N}]\longleftrightarrow
  \{(u^{c\alpha_{1}},\dots,u^{c\alpha_{N}})~|~u\in \C^*\}$$ where $c$
  is any common denominator for the $\alpha_{i}$'s.  Then Lemma
  \ref{le:link} \eqref{linkit:1} implies that the image of the $\C
  ^{*}$ in $G$ is exactly the $\C^*$ in $G_{1}$, since the weights
  differ by a fixed multiple.

  One must still show that image of $D$ is in $G_1$, and (modulo
  $\C^*$) all of $D_1$ is in the image.  So, choose a leaf $w$ in
  $\Gamma_{1}$, and consider corresponding elements in the respective
  discriminant groups, i.e., $e_{w}\in D$ and $\tilde{e}_{w}\in
  D_{1}$.  The following result suffices to complete the proof of
  Theorem \ref{th:compat}.
\end{proof}
\begin{lemma} $\pi(e_{w})\tilde{e}_{w}^{-1} \in \C^*\subset G_{1}$.
\end{lemma}
\begin{proof}
The left hand side is an element in $(\C^*)^{m_{1}}$,
  and we will show that the entry in the slot corresponding to a leaf
  $w'$ in $\Delta_{1}$ is $u^{\tilde{\ell}_{w'}}$ with $u=e^{-2\pi i(\tilde{\ell}_{w}\ell_{v^*v^*}/|D| |D_{1}|^{2})}$.  This would establish the claim.

  To verify the assertion, use the calculation of the entries of
  $e_{w}$ and $\tilde{e}_{w}$ given
  by equation \eqref{eq:1} (just before Definition \ref{def:61}).
For a leaf $w'$, the
  $w'$--entry of $\pi(e_{w})\tilde{e}_{w}^{-1}$ is
  $$e^{2\pi i(-\ell_{ww'}/|D| +\tilde{\ell}_{ww'}/|D_{1}|)}.$$
 By Part 3 of Lemma \ref{le:link}, this expression is as claimed.
\end{proof}

From Theorem \ref{th:compat}, it is clear we have an inclusion
$$G\subset G_{1}\times G_{2} \dots \times G_{k},$$
whence a surjection on the level of characters
$$\hat{G_{1}}\oplus \dots \oplus \hat{G_{k}}\twoheadrightarrow \hat{G}.$$

\comment{would it be possible to put a double arrow at the end of
surjective maps, in the line above, and also in the commutative
diagram below?  WDN Done} Each of the
character groups has its ``reduced weight'' map onto $\Z$ (see
\eqref{eq:exact3}). The map $\hat{G_1}\rightarrow \hat{G}$ induces
a map $\Z\rightarrow \Z$ which is multiplication by the ratio of
the respective reduced weights, i.e.,
$(\ell_w/s)/(\tilde{\ell}_w/s_1)$ for any leaf $w$ in $\Delta_1$.
Let us set $\calD=|D_1||D_2|\dots |D_k|$, $\calD_i=\calD/|D_i|$.
Then  the aforementioned ratio is $s_1\calD_1/s$. Moreover, it is
easy to see that
\begin{equation} s=\operatorname{GCD}(s_1\calD_1,\dots,s_k\calD_k).
\end{equation}
We thus have a commutative diagram of surjections
$$\xymatrix{\hat{G_{1}}\oplus \dots \oplus \hat{G_{k}}\ar@{>>}[r]\ar@{>>}[d]&
  \hat{G}\ar@{>>}[d]\\
\Z\oplus \dots \oplus \Z\ar@{>>}[r]& \Z\hbox to 0 pt{\,\,,\hss}}$$
where the bottom horizontal map is given by dotting with
$(s_1\calD_1/s,\dots,s_k\calD_k/s)$.  In particular, the kernel of
this map is a direct summand.

Recall (Proposition $\ref{prop:M}$) that the characters of $G_{j}$
vanishing on the discriminant group $D_{j}$ form a cyclic group
generated by a distinguished $\hat Q_{j}\in\hat{G_{j}}$, whose reduced
weight is $|D_{j}|/r_{j}$. For each $j$, $2\leq j\leq k$, consider the
$k$--tuple
$$A_{j}:=(\hat Q_{1},0,\dots, 0,-\hat Q_{j}
,0,\dots,0) \quad\in\ \hat{G_{1}}\oplus \dots \oplus \hat{G_{k}}.$$

\begin{theorem}\label{th:ker}
  \begin{enumerate}
  \item The $k-1$ elements $\{A_{j}\}$ form a free basis of the kernel
    of the surjection
    $$\hat{G_{1}}\oplus \dots \oplus \hat{G_{k}}\rightarrow \hat{G}.$$
  \item $r=s$, i.e., the number of branches of the curve
    $C(\Gamma,*)$ equals the GCD of the linking numbers.
  \end{enumerate}
\end{theorem}

\begin{proof} We proceed by induction on the number of nodes in
  $\Gamma$, the case of one node being easily checked using
  Proposition \ref{pr:one}.

  Thus, in the situation at hand, we may assume $r_i=s_i$,
  $i=1,\dots,k$. An element in $\hat{G_{i}}$ (respectively $\hat{G}$)
  is represented by an $m_i$--tuple of integers $I_i$ (respectively,
  $m$--tuple $I$), which are the coefficients of the basic characters
  corresponding to the non-root leaves in $(\Gamma_i,*_i)$ (resp.\
  $(\Gamma,*)$).  Recall that we denote the characters
  themselves by $\hat{I_i}$ and $\hat{I}$; we will denote the image of
  $(\hat{I_1},\dots,\hat{I_k})$ in $\hat{G}$ by $(I_1,\dots
  ,I_k)\hat{~}\,$.
\begin{lemma}\label{le:A} Suppose $(I_1,\dots ,I_k)\hat{~}\,\in \hat{G}$
  has weight 0.  Then for $w$ a leaf in $\Gamma_{i}$, one has an
  equality in $\Q /\Z$:
$$(I_1,\dots ,I_k)\hat{~}\,(e_{w})=\hat{I_i}(\tilde{e}_{w}).$$
\end{lemma}
\begin{proof} We may as well assume $i=1$, and write
  $$I_1=I,\  (I_2,\dots ,I_k)=J,$$
  where $I$ is a tuple of integers indexed by the set $S$ of
  non-root leaves of $\Gamma_{1}$, and $J$ is indexed by the set
  $S'$ of all the other non-root leaves of $\Gamma.$
  The non-reduced weight of $(I_1,\dots ,I_k)\hat{~}\,$ is 0 and is
  computed via linking numbers, yielding
    \begin{equation}
    \label{eq3}
    \sum_{\alpha \in S} i_{\alpha}\ell_{w_{\alpha}}\ +\sum_{\beta
      \in S'} j_{\beta}\ell_{w'_{\beta}}\ =0.
  \end{equation}
  The left hand side of the equation in the Lemma is
  \begin{equation}
    \label{eq4}
    (I_1,\dots ,I_k)\hat{~}\,(e_{w})=
    \operatorname{exp}\Bigl(2\pi i\,\Bigl(\sum_{\alpha \in S}
      i_{\alpha}(e_{w_{\alpha}}\cdot e_{w})\ +\sum_{\beta
        \in S'} j_{\beta}(e_{w'_{\beta}}\cdot e_{w})\Bigr)\Bigr)\,.
  \end{equation}
  By (\ref{le:link}) and \eqref{eq3}, one has
  \begin{align}
    \label{eq5}
    \sum_{\beta \in S'} j_{\beta}(e_{w'_{\beta}}\cdot
    e_{w})=\frac{-1}{|D|}\sum_{\beta \in S'} j_{\beta}\ell_{w'_{\beta}w}
    &=\frac{-c^{2}\ell_{w}}{|D| \ell_{v^*v^*}}
    \sum_{\beta \in S'} j_{\beta}\ell_{w'_{\beta}}\\
    &=\frac{c^{2}\ell_{w}}{|D| \ell_{v^*v^*}}\sum_{\alpha \in S}
    i_{\alpha}\ell_{w_{\alpha}}.
  \end{align}
  We can therefore rewrite equation \eqref{eq4} as
  \begin{equation}
    \label{eq6}
(I_1,\dots ,I_k)\hat{~}\,(e_{w})=
    \operatorname{exp}\Bigl(2\pi i\,\sum_{\alpha \in S} i_{\alpha}\Bigl((e_{w_{\alpha}}\cdot e_{w})
    +\frac{c^{2}\ell_{w}\ell_{w_{\alpha}}}{|D| \ell_{v^*v^*}}\Bigr)\Bigr).
  \end{equation}
  The lemma claims that this expression is equal to
  $\hat{I_1}(\tilde{e}_{w})$, which is
\begin{equation}
  \label{eq7}
  \operatorname{exp}\Bigl(2\pi i\,
  \sum_{\alpha \in S} i_{\alpha}(\tilde{e}_{w_{\alpha}}
  \cdot \tilde{e}_{w})\Bigr).
  \end{equation}
  We need to check for each $\alpha \in S$ equality of the coefficient
  of $i_{\alpha}$.

Using Equation \eqref{eq:1} one needs to check
  that
  $$-\ell_{w_{\alpha}w}/|D|\ +
  c^{2}\ell_{w}\ell_{w_{\alpha}}/|D|
  \ell_{v^*v^*}=-\tilde{\ell}_{w_{\alpha}w}/|D_{1}|.$$ This is a
  simple computation using parts \eqref{linkit:1} and \eqref{linkit:3}
  of Lemma \ref{le:link}.
%
\end{proof}

Continue the proof of Theorem \ref{th:ker}.  Recall
$\hat Q_{j}\in\hat{G}_{j}$ has reduced weight $|D_{j}|/r_{j}=|D_j|/s_j$,
so its non-reduced weight (using linking numbers in
$(\Gamma_{j},*_{j})$) is $|D_{j}|$.  Computing linking numbers in
$(\Gamma,*)$ multiplies this weight by the other $|D_{k'}|$ (Lemma
\ref{le:link} \eqref{linkit:1}), resulting in non-reduced weight
$|D_{1}|\dots |D_{k}|=\calD$ (independent of $j$).  It follows that
the image of $A_{j}$ in $\hat{G}$ has weight $0$.  Since $\hat Q_j$ vanishes
on $D_j$, applying Lemma \ref{le:A}, the image of $A_j$ in $\hat{G}$
vanishes at all $e_{w}$, hence is the trivial character.

Conversely, if $(I_{1},\dots ,I_{k})$ represents the trivial character
of $G$, then it certainly has weight $0$, hence by Lemma \ref{le:A} each
character $\hat I_{j}$ vanishes on the corresponding discriminant
group $D_{j}$.  By Proposition \ref{prop:M}, it follows that
$\hat{I_{j}}$ is equivalent to a multiple $n_j\hat Q_{j}$.  But now the
condition of $\Gamma$--weight equal 0 means that
$\sum_{j=1}^{k}n_{j}=0$.  That the $A_{j}$ form a basis of this space
is now easy.  This completes the proof of the first assertion of the
Theorem.

We now have a commutative diagram of short exact sequences:
 $$\xymatrix{0\ar[r]& \bigoplus_{j=2}^k \Z\cdot A_k\ar[r]\ar[d]&
\hat{G_{1}}\oplus \dots \oplus \hat{G_{k}}\ar[r]\ar[d]& \hat{G}\ar[r]\ar[d]& 0\\
0\ar[r]& K\ar[r]&\Z\oplus \dots \oplus \Z\ar[r]&
\Z\ar[r]& 0\hbox to 0pt{~.\hss}\\}$$

The right two vertical maps are surjective, with kernels finite groups
of order $r_1\dots r_k$ and $r$, respectively.  The left vertical map
is injective, and the order of the cokernel is (since $K$ is a direct
summand) the order of the torsion subgroup of the quotient of $\Z^k$
by the image of the $A_j$'s.  So, one considers a $(k-1)\times k$
matrix whose non-$0$ entries are of the form $\pm |D_i|/s_i$. The
maximal minors are of the form $s_j\calD_j/(s_1\dots s_k)$; so the
order of this group is the GCD of these minors, which is $s/s_1\dots
s_k$.  The snake sequence for the diagram then yields that $r=s$.
This completes the proof of the Theorem.
\end{proof}

We  restate the key part of the Theorem:

\begin{corollary}\label{cor:ch}  Every element of $\hat{G}$ may be
  written in the form
$$X_1+X_2+\dots\ +X_k,\quad
X_i\in \hat{G_{i}}\,,$$ and this representation is unique
modulo the equations $\hat Q_{i}=\hat Q_{j}$, all $i,j$.\qed
\end{corollary}

\begin{corollary}\label{cor:eq} Consider as above the curves
  $C=C(\Gamma,*)$ and $C_i=C(\Gamma_i,*_i)$ coming from $\Gamma$ and
  its rooted subtrees $\Gamma_i$, $i=1, \dots,k$.  For each $i$, let
  $Y_{i,j}$ be variables corresponding to the non-root leaves of
  $\Gamma_i$, and let $\calJ_i$ be the ideal in the variables
  $Y_{i,j}$ defining $C_i$.  Assume that each $\hat Q_i\in \calS_i$,
  so we can choose a non-negative $m_i$--tuple $Q_i$ representing
  $\hat Q_i$, and hence a monomial $Y_i^{Q_i}$ in the variables
  $Y_{i,j}$.

Then the ideal defining $C$ in all the variables
  $Y_{i,j}$ is generated by $\calJ_1,\calJ_2,\dots,\calJ_k$ and
  $Y_1^{Q_1}-Y_i^{Q_i}$, $i=2,\dots,k.$
\end{corollary}
\begin{proof}
  According to Proposition \ref{prop:G}, the ideal of $C$ (resp.\
  $C_i$) is generated by all differences of pairs of monomials whose
  (non-negative) exponents have the same image in the character group
  $\hat{G}$ (resp.\ $\hat G_i$).  Suppose $I=(I_1, \dots,I_k)$ and
  $I'=(I'_1,\dots,I'_k)$ are non-negative exponent tuples giving the
  same element of $\hat{G}$.
  We will first simplify the relation $Y^I=Y^{I'}$ modulo the ideals
  $\calJ_i$.

  According to Theorem \ref{th:ker}, the images of $I$ and $I'$ in
  $\hat{G_{1}}\oplus \dots \oplus \hat{G_{k}}$ differ by a sum
  $\sum_{j=2}^kp_jA_j$ of the $A_j$'s.  Putting $-p_1=\sum_{j=2}^k
  p_j$, this means
$$\hat{I_i}-\hat{I'_i}=-p_i\hat Q_i \text{ in }\hat G_i,\  i=1,\dots,k,
\,.$$ For each $p_i \geq 0$, the tuples $I_i+p_iQ_i$ and $I'_i$ are
non-negative tuples giving the same element in $\hat{G}_i$.  So,
modulo relations coming from $\calJ_i$, one can subtract $I_i$ from
each exponent tuple $I,I'$ in the $i$--slot, leaving $0$ and $p_iQ_i$
in the new $i$--slots.  For $p_i<0$ one subtracts instead $I'_i$
modulo relations from $\calJ_i$ to get $-p_iQ_i$ and $0$ respectively
in these slots. After doing this for each $i$ one has either $0$ or a
positive multiple of $Q_i$ in each slot of both $I$ and $I'$ and the
total coefficient sum is the same for each, namely
$p:=\sum_{p_i>0}p_i$.  At this point we have simplified our relation
$Y^I=Y^{I'}$ to the point where each side of it is equivalent to
$Y_1^{pQ_1}$ using the relations $Y_1^{Q_1}=Y_i^{Q_i}$. This proves
the corollary.\end{proof}

\section{semigroups and delta-invariants of $D$--curves from
$(\Gamma,*)$}
\label{sec:nu}

Let us maintain the same basic setup as the last section: the rooted
tree $(\Gamma,*)$; the discriminant group $D=D(\Gamma)$, of order
$|D|$; the group $G$ and its character group $\hat{G}$; the $D$--curve
$C=C(\Gamma,*)$; the reduced weight $|X|$ of an element $X$ of
$\hat{G}$, giving a surjection $\hat{G}\rightarrow \Z$, with kernel of
order $r$; a surjection $\hat{G}\rightarrow \hat{D}$, with kernel the
cyclic group generated by a distinguished element $Q$ of weight $|D|
/r$.  $r$ is both the number of components of the curve and the GCD of
the non-reduced weights (i.e., linking numbers) of the leaves.
Non-negative characters of $G$ give a semigroup $\calS \subset
\hat{G}$, which is the value semigroup of the corresponding
$D$--curve.  We are interested in the $\delta$--invariant of the curve
(which is the number of gaps of the semigroup $\calS$), and the
question of whether $Q\in \calS$.  All elements of $\hat{G}$ of
sufficiently high weight are in $\calS$, while there are $r-1$ gaps of
weight 0.

As before, we compare data of $(\Gamma,*)$ with those of the subtrees
$(\Gamma_{i},*_{i})$.  Recall we have defined
$$  \calD := |D_1||D_2| \dots |D_k|,\ \calD_i:= \calD/|D_i|.$$
Corollary \ref{cor:ch} indicates that every element of $\hat{G}$ may
be written
$$X_{1}+X_{2}+\dots\ +X_{k},$$
where $X_{i}\in \hat{G_{i}}$, and this representation is unique modulo
all the equations $\hat Q_{i}=\hat Q_{j}$.  Using $|X_{i}|_{i}$ to
denote weight in $\hat{G_{i}}$, one easily checks that
\begin{equation}
  \label{eq:5}
  |X_{i}|= (r_{i}/r)\calD_i|X_{i}|_{i}.
\end{equation}
(Recall every $\hat Q_{i}$ has the same weight $\calD/r$ in $\hat{G}$.)

\begin{theorem}\label{th:delta}
  Consider the data $(\Gamma,*)$, $\hat{G}$, $\hat Q$, $|D|$, $r$, $\calS$,
  $\delta$ as before, with corresponding notation for the subtrees.
  Then:
  \begin{enumerate}
  \item The number of gaps of the semigroups satisfies
   $$2\delta -r \leq  \sum_{i=1}^{k}\calD_i(2\delta_{i}-r_{i})\
   +\ (k-1)\calD.$$
  \item If equality holds, then $\hat Q_{i}\in \calS_{i}$, for all $i$.
  \end{enumerate}
\end{theorem}
\begin{proof}
  To show the inequality, we count the gaps in $\calS$.  We write an
  element of $\hat{G}$ as $X=X_{1}+\dots +X_{k}$, where $X_{i}\in
  \hat{G_{i}}$, and the representation is unique up to repeated
  alterations of the form: add $\hat Q_{i}$ to $X_{i}$ and subtract
  $\hat Q_{j}$ from $X_{j}$ for some $i,j$.  Denote
  $$q_i:=|\hat Q_{i}|_{i}=|D_i|/r_{i}.$$ We claim that every $X\in
  \hat{G}$ of weight $\geq 0$ has a representation $X=X_{1}+\dots
  +X_{k}$ in one of the following classes:

\begin{enumerate}
\item[$C_0$:] $|X_{1}|_{1}<0$; $0\leq |X_{i}|_{i}<q_{i}$ for $i\geq
  2$ (and $|X|\ge 0$ by assumption).

\item[$C_1$:] $0\leq|X_{1}|_{1}$ and $X_{1}\notin \calS_{1}$; $0\leq
    |X_{i}|_{i}<q_{i}$ for $i\geq 2.$

\item[$C_2$:] $X_{1}\in \calS_{1}$ and $X_{1}-s\hat Q_{1} \notin \calS_{1}$ for
    $s\geq 1$; $0\leq|X_{2}|_{2}$ and $X_{2}\notin \calS_{2}$; $0\leq
    |X_{j}|_{j}<q_{j}$ for $j\geq 3$.

\item[\dots] ~\dots \dots \dots\\[-10pt]

\item[$C_{i}$:] For $1\leq j \leq i-1$, $X_{j}\in \calS_{j}$ and
    $X_{j}-s\hat Q_{j}\notin \calS_{j}$ for $s\geq 1$; $0\leq |X_{i}|_i$ and
    $X_{i}\notin \calS_{i}$; $0\leq
    |X_{j}|_{j}<q_{j}$ for $j\geq i+1$. 

\item[\dots] ~\dots \dots \dots\\[-10pt]

\item[$C_{k}$:] For $1\leq j \leq k-1$, $X_{j}\in \calS_{j}$ and
    $X_{j}-s\hat Q_{j}\notin \calS_{j}$ for $s\geq 1$; $0\leq |X_{k}|_k$ and
    $X_{k}\notin \calS_{k}$.

\item[$C_{k+1}$:] For $1\leq j \leq k-1$, $X_{j}\in \calS_{j}$ and
    $X_{j}-s\hat Q_{j}\notin \calS_{j}$ for $s\geq 1$; $X_{k}\in \calS_{k}$.
\end{enumerate}

\par\smallskip\noindent
\comment{WDN: You found this paragraph hard to
  understand before. I hope this version is better.}%
To see this, given $X$ with $|X|\ge 0$, start by arranging the weights
of $X_i$, $i\geq 2$, to be less than the weight $q_i$ of the
corresponding $\hat Q_{i}$.  If the representation is then not in the
class $C_0$ or $C_1$, then $X_1\in \calS_1$. So subtract
$\hat Q_1$'s as necessary from $X_1$ and simultaneously add
$\hat Q_2$'s to $X_2$ until $X_1\in\calS_1$ but
$X_1-\hat Q_1\notin\calS_1$. If now the representation is not in $C_2$,
then $X_2\in \calS_2$, so subtract $\hat Q_2$'s as necessary from $X_2$
and add  $\hat Q_3$'s to $X_3$ to make $X_2\in\calS_2$
but $X_2-\hat Q_2\notin\calS_2$.  Repeat this procedure until the
representation is in some $C_i$. If one fails all the way to $C_k$,
then the final representation is in $C_{k+1}$.

Elements in $C_{k+1}$ are clearly not gaps, so adding up the
sizes of the  other classes gives an upper bound for
$\delta$, the number of gaps.  Recall that the number of elements of
$\hat{G_{i}}$ of a given weight is equal to $r_{i}$.  But
$q_{i}r_{i}=|D_i|$, so the number of elements in class $C_1$ is
$\delta_{1}\calD_1$.  Next, every $(\hat Q_{i})$--coset in $\hat{G_{i}}$
contains a unique minimal representative in $\calS_{i}$, i.e., which is
not in $\calS_{i}$ after subtracting any positive multiple of
$\hat Q_{i}$. Thus, there are $[\hat{G_{i}}:(\hat Q_{i})]=|\hat D_{i}|=|D_i|$
such elements.  So, class $C_j$ contains $\delta_{j}\calD_j$ elements
for $j=2,\dots, k$.  Thus, $\delta$ is bounded above as follows:
$$\delta\le |C_0|+\,\sum_{i=1}^{k}\delta_{i}\calD_i\,.$$

To count elements of $C_0$, we need to know the number $N$ of
allowed degrees $x_i:=|X_{i}|_{i}$; then the total count of these
elements would be $r_{1}\dots r_{k} N,$ where (thanks to (\ref{eq:5}))
$$N=\#\{(x_{1},\dots,x_{k})\in \Z^{k}|\ x_{1}<0;\ 0\leq x_{i}
<q_{i},\ \text{all}\  i\geq 2; \sum_{i=1}^{k}x_{i}/q_{i}\ \geq 0\}.$$
We use (but prove later) the following
\begin{lemma}\label{le:count} Let $\calQ=q_{1}\dots q_{k},\
  \calQ_{i}=\calQ/q_{i}.$ Then
    $$2N=(k-1)\calQ\ -\sum_{i=1}^{k}\calQ_{i}\ +h,$$
where $h=\operatorname{GCD}(\calQ_{1},\dots,\calQ_{k}).$
\end{lemma}

Using the Lemma we conclude that
$$\delta \leq \sum_{i=1}^{k}\delta_{i}\calD_i
+\frac12\, r_{1}\dots r_{k}\Bigl((k-1)\calQ\ -\sum_{i=1}^{k}\calQ_{i}\
+h\Bigr)\,.$$ But $r=\operatorname{GCD}(r_i\calD_i)$, while $r_1\dots
r_k\calQ_i=r_i\calD_i$, so
$$r=r_{1}\dots
r_{k}\cdot\operatorname{GCD} (\calQ_{1},\dots,\calQ_{k})\ =r_{1}\dots
r_{k}h.$$ The inequality in the theorem follows from these.

If equality holds, then the classes $C_0,C_1,\dots, C_k$ consist only
of gaps, and there is no overlap between them.  Suppose $\hat Q_{i}\notin
\calS_{i}$ for some $i$.  Then $X=0+\dots+\hat Q_i+\dots +0$ is in $C_i$,
hence a gap.  It is equal to $0+\dots+\hat Q_{j}+\dots +0$ for each $j$, so
$\hat Q_j \notin \calS_j$, so we have a common element in all the classes
$C_1,\dots,C_k$.  This contradicts that the $C_i$ don't overlap,
which proves the second claim of the theorem.
\end{proof}

\begin{proof}[Proof of Lemma \ref{le:count}]
  To prove the Lemma, consider a one node resolution diagram with $-b$
  in the center, and $k+1$ length $1$ strings emanating, with weights
  $-q_{1},\dots,-q_{k}$ and $-b'$.  Here $b$ is any integer at least
  $k+1$, and $b'>0$ is arbitrary.  View the $-b'$ vertex as the root.
  This gives a $(\Gamma,*)$, and $k$ subdiagrams $(\Gamma_{i},*_{i})$,
  each of which consists of only one vertex.  So, each $G_{i}=\C^*$,
  $r_{i}=1$, $\delta_{i}=q_{i}$, $\calS_{i}=\N$, $\hat{G}_{i}=\Z$. The
  gaps of $\hat{G}$ can be counted as above, but they are exactly
  those of the first type $C_0$.  Thus, the desired quantity $N$ is
  exactly the number of gaps $\delta$ (i.e., the delta-invariant of
  the corresponding curve).  But the curve is the complete
  intersection defined by $X_{1}^{q_{1}}=X_{i}^{q_{i}}, \
  i=2,\dots,k,$ as in Example \ref{ex:instr} and Proposition
  \ref{pr:one}.  This curve has $s=h$ branches, using the notation
  above.  Using Example \ref{ex:instr} gives
$$2N=(k-1)\calQ\ -\sum_{i=1}^{k}\calQ_{i}+h.$$  This proves the lemma.
\end{proof}

There is another invariant of $(\Gamma,*)$ whose relationship to those
of its subdiagrams $(\Gamma_i,*_i)$ is similar to the situation for
$2\delta -r$, as revealed by Theorem \ref{th:delta}.  Namely,
let
$$\nu(\Gamma,*):= \sum_{v\neq *}(d_v-2)\ell_{*,v}\,,$$
the sum being over all vertices of $\Gamma$ except for $*$.  The
following statements are easy to verify:
\begin{enumerate}
\item $\nu$ depends only on the splice diagram $(\Delta,*).$
\item $\nu(\Gamma,*)=\sum_{i=1}^k \calD_i\nu_i\ +\ (k-1)\calD.$
\item If $(\Gamma,*)$ has one node (Proposition \ref{pr:one}), then
  $\nu=2\delta -r$.
\end{enumerate}

This brings us to our major result.

\begin{theorem}\label{th:maj} Let $C=C(\Gamma,*)$ be the $D$--curve
  constructed from a rooted resolution diagram.  Then
  $2\delta(C)-r\leq \nu(\Gamma,*)$, and equality implies the
  following:
\begin{enumerate}
\item $C$ is a complete intersection curve
\item The splice diagram $\Delta$ satisfies the semigroup and
  congruence conditions at any node
  in the direction away from the leaf $*$.
\end{enumerate}
\end{theorem}
We recall that the semigroup and congruence conditions at a node can
be formulated that for all outgoing edges at the node, monomials of
appropriate weight can be found which transform equivariantly (i.e.,
with the same character) with respect to the $D$-action.
\begin{proof} We do induction on the number of nodes of $\Gamma$.  For
  one node, the inequality in question is an equality, and the claims
  follow from Example \ref{ex:instr} (6) and Proposition \ref{pr:one}.
  The semigroup and congruence conditions are automatic.

  In the general case, we as usual compare $(\Gamma,*)$ with its
  subtrees $(\Gamma_i,*_i)$, $i=1,\dots,k$. Theorem \ref{th:delta}(1),
  the second statement about $\nu$, and induction give the general
  inequality.  Also, equality for $(\Gamma,*)$ implies equality for
  all the $(\Gamma_i,*_i)$, and that each $Q_i \in \calS_i$.
  The induction assumption says that the curves $C_i$ are complete intersections;
  applying Corollary \ref{cor:eq}, it follows that $C$ is as well
  (compare number of defining equations to number of variables).
  Further, the monomials involved in the added equations
  $Y_1^{Q_1}-Y_i^{Q_i}$, $i=2,\dots,k$ are monomials whose characters in
  $\hat G$ are equal by Theorem \ref{th:ker} and have the correct weight.
  This gives the semigroup and congruence condition at the node
  closest to $*$ in the directions away from $*$. For nodes further
  from $*$ it was proved during the induction, except that
  equivariance was proved in a subgraph $\Gamma_i$ and hence proved
  for $D_i$ rather than $D$. But $D$--equivariance follows from the
  fact the the $D$--action induces the $D_i$--action on the variables
  coming from the subgraph $\Gamma_i$ (Theorem \ref{th:compat}).
\end{proof}

\section{Proof of the End Curve Theorem}\label{sec:proof}

Let $(X,o)$ be a normal surface singularity with $\Q$HS link
$\Sigma$. Recall that we say a knot or link $K\subset \Sigma$ is
\emph{cut out by the
  analytic function $z\colon (X,o)\to (\C,o)$} if the pair $(\Sigma,
K)$ is topologically the link of the pair $(X,\{z=0\})$ (i.e., the
reduced germ $(X,\{z=0\},o)$ is homeomorphic, preserving orientations,
to the cone on $(\Sigma,K)$). In \cite{neumann81} it is shown that the
link of a surface--curve germ pair $(X,B,o)$ determines the minimal
good resolution of this pair.\comment{WDN: I don't see why the 
  stuff in small font is
  needed here. I'd prefer to get rid of them}
\Omit{$K$ is a knot (i.e., is connected)
 exactly when the principal ideal of $z$ in the local ring of $X$
 has only one associated prime, i.e. when the reduced zero-set $B$ is
 an irreducible curve.

 Suppose from now on that $\Sigma$ is a $\Q$HS.  Recall that $z$ is an
 \emph{end-curve function} if its zero-set on the MGR $(\tilde X,E)$
 is $Z+mC$, where $Z=\sum a_v E_v$ is an effective exceptional divisor
 and $C$ is a smooth curve intersecting $E$ once, transversally along
 a curve $E_1$ corresponding to a leaf.  Then $C\cap \Sigma =K$ is an
 \emph{end-knot}.  Since $Z+mC$ is numerically trivial (i.e., dots to
 0 with every $E_i$), it follows that $Z\cdot E_j=0, j\neq 1$ and
 $Z\cdot E_1=-m.$ But if $d$ is the order of the image of the dual
 basis element $e_1$ of $E_1$ in the discriminant group, then there is
 an effective divisor $Z_1$ dotting to 0 with all $E_j$ except $E_1$,
 with which it dots to $-d$; and, $Z$ is some integral multiple
 $m'Z_1$.  In particular, $m=m'd$ and $Z+mC=m'(Z_1+dC)$.  Since the
 divisor of $z$ on the simply connected space $\tilde X$ is a multiple
 of $m'$, there exists a holomorphic function $z'$ on $\tilde X$ (and
 hence on $X$) which is an $m'$--th root of $z$.  In particular, we may
 assume that the end-curve function $z$ vanishes to order exactly $d$
 along the curve $C$, or along the image curve $B$ in $X$.  Note also
 that $d$ is the order of the class of the end-knot $K$ in
 $H_1(\Sigma)$ \cite[Corollary 12.11]{nw1}.}

Let $v_i$, $i=1,\dots,t$ be the leaves of the resolution graph
$\Gamma$.   Suppose that for each $i$ we have
an end-curve function $z_i$, vanishing\comment{WDN deleted: ``(the
  minimal number)''; the proof was explicitly written to not require
  this.} $d_i$ times along the\comment{WDN deleted ``(irreducible)''}
end-curve $B_i \subset X$\comment{WDN deleted:; we could say that
  \emph{the end-knots are algebraic} since there is already other
  terminology for this in the literature}. The following lemma tells
us that a $d_i$--th root $x_i$ of $z_i$ is a well-defined analytic
function on the \UAC{} $(V,0)$ of $(X,o)$, and that
$x_i$ vanishes to order $1$ on its zero-set.
\begin{lemma}\label{le:lift}
  Let $z\colon (X,o)\to (\C,0)$ be an analytic function that vanishes
  to order $d$ on its reduced zero set $B\subset X$. Then the
  multivalued function $z^{1/d}$ on $X$ lifts to a single valued
  function $x$ on the \UAC{} $(V,0)$ of $(X,o)$, and
  $x$ vanishes to order $1$ on its zero set (there are $d$ such lifts
  that differ by $d$--th roots of unity). If $B$ is irreducible then the
  zero-set of $x$ has $|D|/d'$ components, where $d'$ is the order of
  the the class of $K$ (the link of $B$) in $H_1(\Sigma;\Z)$.
\end{lemma}
\begin{proof}
  Take the branched cover $X'\rightarrow X$ given on the local ring
  level by adjoining $t$ satisfying $t^d-z=0$, and then normalizing.
  At any point of $B-\{o\}$, choose local analytic coordinates $u,v$,
  with $B$ given by $u=0$; we may further assume that locally $z=u^d.$
  Over such a point, $X'$ is given by normalizing $t^d-u^d=0$,
  yielding a smooth and unramified $d$-fold cyclic cover.  Thus,
  $X'\to X$ is unramified away from the singular point.  Clearly, $t$
  is still a single-valued function $z'$ on this cover, and $z'$
  vanishes to order $1$ on its zero-set.  $z'$ is well defined up to
  the covering transformations of $X'\to X$, which multiply $z'$ by
  $d$--th roots of unity. 

  We may assume $X'$ is connected (if $X'$ has $k$ components then
  replace $X'$ by one of its components, $z$ by $(z')^{d/k}$, which is
  then well defined on $X$, and $d$ by $d/k$).  Since $X'\to X$ is a
  connected abelian cover, it is covered by the universal abelian
  cover $V\to X$ so we get our desired lift $x$ of $z$ to $V$ by
  composing $z'$ with the projection $V\to X'$.  

  Now assume $B$ is connected and $K$ is its link. Let $d'$ be the
  order of the class of $K$ in $H_1(\Sigma)$. Then each component of
  the inverse image $\tilde K$ of $K$ in $\ab\Sigma$ is an $d'$-fold
  cover of $K$, so there are $|D|/d'$ such components. Since $\tilde
  K$ is the link of $\{x=0\}$, the final sentence of the lemma follows.
  (Although we don't need
  it, it is not hard to see that $k$ above is $d/d'$, so
  there is a $d/d'$--th root of $z$ that is defined on $X$. Hence,
  $d'$ is the least order of vanishing any function $z$ as in the
  lemma.)\end{proof}

Denote the zero set of $x_i$ in $V$ by $C_i$. This is a $D$-curve,
where $D=D(\Gamma)$ is the discriminant group, and it has $|D|/d_i$
branches, where $d_i$ is the order of the class of the end-knot $K$ in
$H_1(\Sigma)$. By \cite[Corollary 12.11]{nw1}, $|D|/d_i=r_i$.  We will
concentrate for the moment on $C_1$.  By Lemma \ref{le:chi of fiber}
below, its $\delta$--invariant is
\begin{equation}
  \label{eq:delta_of_C1}
  \delta(C_1)=\frac12\bigl(r_1+\nu(\Gamma,v_1)\bigr)= \frac12\bigl(r_1+\sum_{v\neq v_1}(d_v-2)\ell_{v_1,v}\bigr)\,.
\end{equation}
Let $\hat G$ be the character group associated with $C_1$ and its
associated graded, as in section \ref{sec:more}
, and let $\calS\subset \hat G$ be the value semigroup. The
characters in $\hat G$ associated with the functions
$x_2,\dots,x_t$ generate a subsemigroup $\calS_0$ of $\calS$,
whence
\begin{equation}
\label{eq:ineq} \delta(\calS_0)\ge\delta(\calS)=\delta(C_1).
\end{equation}

\comment{JMW--not sure where to put the following}
\Omit{\par
The $\delta$-invariant of the subsemigroup $\mathcal
S_0$ is equal to the $\delta$-invariant of the rooted resolution
diagram $(\Gamma,v_1)$ as in Section \ref{sec:curves from
diagrams}.
\begin{proof} Since intersection number of curves in $X$ is given by
  linking number of their links in $\Sigma$, the function $x_i$
  ($i>1)$ has vanishing degree $\ell_{v_1v_i}$ on the curve $C_1$,
  hence $\ell_{v_1v_i}/r_1$ on each branch; this last quantity is thus
  the weight of $x_i$ in the associated graded of $C_1$
  (alternatively, this can be seen by considering the intersection
  number of the proper transform of $C_1$ with the zero set of $x_i$
  in the resolution).  By
  Theorem \ref{th:char} below\comment{JMW--put the proof in last section, and
  change title to ``Some topological computations''?}, $D$ acts on
  $x_i$ via the character $\chi_i$ corresponding to the leaf $v_i\in
  \Gamma$.  Thus, as in Proposition \ref{prop:G}, $\delta(\calS_0)$ is
  the sum over all $n\geq 0$ of $r_1$ minus the number of monomials
  $x_2^{j_2}\cdots x_t^{j_t}$ of weight $\sum_{k=2}^{t}j_k
  \ell_{v_1v_k}/r_1\ =n$ and distinct $D$-characters. On the other
  hand, this number also computes the $\delta$-invariant of the curve
  $C(\Gamma,v_1)$, as proved in Section \ref{sec:curves from diagrams}
  and summarized at the beginning of Section \ref{sec:nu}.
\end{proof}
}

\begin{lemma} The subsemigroup $\mathcal S_0$ is equal to the value
  semigroup of the curve associated to the rooted resolution diagram
  $(\Gamma,v_1)$ as in Section \ref{sec:curves from diagrams}. In
  particular, their $\delta$--invariants are equal.
\end{lemma}
\begin{proof} Since intersection number of curves in $X$ is given by
  linking number of their links in $\Sigma$, the function $x_i$
  ($i>1)$ has vanishing degree $\ell_{v_1v_i}$ on the curve $C_1$,
  hence $\ell_{v_1v_i}/r_1$ on each branch; this last quantity is thus
  the weight of $x_i$ in the associated graded of $C_1$
  (alternatively, this can be seen by considering the intersection
  number of the proper transform of $C_1$ with the zero set of $x_i$
  in the resolution).  By
  Theorem \ref{th:char} below\comment{JMW--put the proof in last section, and
  change title to ``Some topological computations''?}, $D$ acts on
  $x_i$ via the character $\chi_i$ corresponding to the leaf $v_i\in
  \Gamma$.  Thus, the $\hat G$--character of $x_i$ is the generator
  $\overline\chi_i$ of $\calS(C(\Gamma,v_1))$ as described in section
  \ref{sec:curves from diagrams}. 
\end{proof}


From the Lemma and  Theorem \ref{th:maj}, we
have
\begin{equation}
  \label{eq:delta_ineq}
  \delta(\calS_0)\leq\frac12\bigl(r_1+\nu(\Gamma,v_1)\bigr).
\end{equation}
Comparing \eqref{eq:delta_of_C1}, \eqref{eq:ineq}, \eqref{eq:delta_ineq}, the
inequalities are actually equalities.
Hence, $\calS_0=\calS$ and this is the semigroup both for the
associated graded of $C_1$ and the model curve $C(\Gamma,v_1)$, so
these are isomorphic as $D$-curves (Corollary
\ref{cor:classification}).  Again by Theorem \ref{th:maj}, each curve
is a complete intersection with maximal ideal generated by
$x_2,\dots,x_t$.  Since this curve is the associated graded of $C_1$,
it follows that $C_1$ is a complete intersection and $x_2,\dots,x_t$
generates its maximal ideal. We conclude that $(V,0)$ is a complete
intersection with maximal ideal generated by $x_1,\dots, x_t$.

Moreover, Theorem \ref{th:maj} gives us the semigroup and
congruence conditions at all nodes in the directions away from the
leaf $v_1$. Repeating the argument at all leaves gives all semigroup
and congruence conditions.

It remains to show that $V$ is defined by a system of $D$-equivariant
splice equations using the functions $x_1,\dots, x_t$. Pick a $v$ of
$\Gamma$ of valency $\delta$. Denote by $E_v$ the exceptional curve
corresponding to $v$ and by $E_1, \dots, E_\delta$ the $\delta$
exceptional curves that intersect $E_v$. Since the congruence and
semigroup conditions are satisfied, we can find a system of admissible
monomials $M_1,\dots, M_\delta$ of (unreduced) weight $\ell_{vv}$,
corresponding to the outgoing edges at $v$, which transform the same
way with respect to the action of $D$. For $1\le i<\delta$ the ratio
$M_i/M_\delta$ is thus $D$-invariant, hence defined on $X=V/D$. In the
proof of Theorem 10.1 of \cite{nw1} (see also Theorem 4.1 of
\cite{nw2}) it is shown that this function is meromorphic on the
exceptional curve $E_v$ corresponding to $v$, with a simple zero at
the point $E_v\cap E_i$, a simple pole at $E_v\cap E_\delta$ and no
other zeroes or poles. It follows that there are $\delta-2$ linear
relations among the $M_i/M_\delta$ on $E_v$. If
$a_1M_1/M_\delta+\dots+a_\delta M_\delta/M_\delta=0$ is one of these linear
relations, write $L:=a_1M_1+\dots+a_\delta M_\delta$. Then the order of
vanishing of $L^\delta$ on $E_v$ is greater than that of
$M_\delta^d$. Since the $v$-weight of a function $f$ on $X$ is
measured by the order of vanishing of $f^d$ on $E_v$, we see that $L$
has $v$--weight greater than $\ell_{vv}$, so that, as in
the proof of Theorem 10.1 of \cite{nw1}, we can adjust $L$ by
something that vanishes to higher $v$-weight to get an equation of
splice type that holds identically on $V$. Doing this for each of our
linear relations and repeating at all nodes gives a system of splice
equations that hold on $V$. As proved in \cite{nw1}, such a system of
splice type equations defines a complete intersection singularity
$(V',0)$ whose $D$-quotient $(X',o)$ has resolution graph
$\Gamma$. Moreover since the local rings are subrings of the local
rings of $V$ and $X$, we have finite maps $V\to V'$ and $X\to X'$.
The degree of the map $V\to V'$ can be computed by restricting to the
curve $C_1=\{x_1=0\}$ and then taking the associated graded of this
curve. By Corollary \ref{cor:eq} we see this way that the degree is
$1$, so the proof is complete.\qed

\section{Topological computations.}\label{sec:top}

\subsection{Linking numbers}
In this subsection we describe the interpretation of the numbers
$\ell_{ij}$ as linking numbers. Recall first that if $K_1$ and
$K_2$ are disjoint oriented knots in a $\Q$HS $\Sigma$ then their
\emph{linking number} is defined as follows: Some multiple $dK_1$
bounds a $2$--chain $A$ in $\Sigma$ and $\ell(K_1,K_2)$ is defined
as $\frac1d A\cdot K_2\in\Q$ (intersection number). A (standard)
easy calculation shows that this is well-defined. Now suppose
$\Sigma$ bounds an oriented $4$--manifold $Y$ with $H_2(Y;\Q)=0$.
Then multiples $d_1K_1$ and $d_2K_2$ bound $2$--chains $A_1$ and
$A_2$ in $X$ and $\ell(K_1,K_2)$ can be computed as
$\frac1{d_1d_2}A_1\cdot A_2$ (see, e.g. Durfee \cite{durfee}; the
point is that it is again easy to see this is independent of
choices, and if one chooses $A_1$ to lie in $\Sigma$ and $A_2$ to
be transverse to $\Sigma$ one gets the previous definition). We
can extend to the case that $H_2(Y)\ne 0$ by requiring that $A_1$
be chosen to have zero intersection with any $2$--cycle (i.e.,
closed $2$--chain) in $Y$; it clearly suffices to require this for
$2$--cycles representing a generating set of $H_2(Y)$. Again, the
proof that this works only involves showing that it gives a
well-defined invariant, which is as before.

Suppose now that $\Sigma=\partial X$ is a $\Q$HS singularity link and
$Y\to X$ is a resolution of the singularity. Let $K_v$ and
$K_w$ be knots in $\Sigma$ represented by meridians of exceptional
curves $E_v$ and $E_w$ in $Y$. 
\begin{proposition}
  \label{prop:linking}
  $\ell(K_v,K_w)=\frac1{|D|}\ell_{vw}$.
\end{proposition}
\begin{proof}
  Let $D_v$ and $D_w$ be transverse disks to $E_v$ and $E_w$ with
  boundaries $K_v$ and $K_w$. Recall that the matrix $(\ell_{ij})$ is
  $-|D|(E_i\cdot E_j)^{-1}$ (see Def.\ \ref{def:linking}). It follows
  that a 2-chain $A$ whose boundary is $|D|K_v$ and which dots to zero
  with each $E_i$ is given by $A=|D|D_v+\sum_i\ell_{vi}E_i$. So
  $\ell(K_v,K_w)=\frac1{|D|}A\cdot D_w=\frac1{|D|}\ell_{vw}$.
\end{proof}
\subsection{Torsion linking form}
The torsion linking form is a non-degenerate bilinear $\Q/\Z$-valued
pairing on the torsion of $H_1(M;\Z)$ for any closed oriented
$3$--manifold $M$. We recall the definition. If $\alpha,\beta\in
H_1(\Sigma;\Z)$ are torsion elements, we represent $\alpha$ and
$\beta$ by disjoint $1$--cycles $a$ and $b$. Since some multiple $da$
of $a$ bounds, so we can find a $2$--chain $A$ with $\partial
A=da$. Then $\ell(\alpha,\beta):=\frac1d\,A\cdot b\in\Q/\Z$ where
$A.b$ is the algebraic intersection number.

For a $\Q$HS $\Sigma$ the linking form is a non-degenerate
symmetric bilinear pairing $$\ell\colon H_1(\Sigma;Z)\times
H_1(\Sigma;\Z)\to \Q/\Z$$ and hence gives an isomorphism of
$D=H_1(\Sigma;Z)$ with its character group $\hat
D=\Hom(D,\C^*)\cong\Hom(D,\Q/\Z)$. We take the negative of this
isomorphism and for $x\in D$ we call the character $e^{-2\pi
i\ell(x,-)} \in\hat D$ the \emph{character dual to $x$}. By the
above Proposition \ref{prop:linking}, we have:
\begin{proposition}
  For a $\Q$HS link of a complex surface singularity, the torsion
  linking form is the negative of the form $e_v.e_w$ of Section
  \ref{subsec:action}.\qed
\end{proposition}

We now want to return to the situation of Lemma \ref{le:lift}, where
we lift a root of an end-curve function on our singularity $(X,o)$ to a
function on the \UAC. It is convenient now to
restrict just to the link of the singularity (and of the curve
$B\subset X$). The result we need, Theorem \ref{th:char} below, is a
general statement about rational homology spheres.

So we assume we have a $\Q$HS $\Sigma$ and a knot (or link) $K\subset
\Sigma$ and a smooth function $z\colon \Sigma\to \C$ which vanishes to
order exactly $d$ along $K$. The proof of Lemma \ref{le:lift} applies
to see that the multivalued function $z^{1/d}$ on $\Sigma$ can be
lifted to a single-valued function $x$ on the universal abelian cover
$\ab\Sigma$ which vanishes to order 1 on its zero set (i.e., 0 is a
regular value). The covering transformation group for the universal
abelian covering $\pi\colon\ab \Sigma\to \Sigma$ is
$D=H_1(\Sigma;\Z)$.
\begin{theorem}\label{th:char}
  Let $K\subset \Sigma$ and $x$, as above, a lift to $\ab\Sigma$ of
  the $d$--th root of a function $z$ that vanishes to order $d$ along $K$.
  Then the action of $D$ on $\ab\Sigma$ transforms the function $x$ by
  the character dual to the homology class $[K]\in
  D=H_1(\Sigma;\Z)$. That is,
$$x(hp)=e^{-2\pi i\ell([K],h)}x(p)\quad \text{for }p\in
  \ab\Sigma \text{ and }h\in D\,.$$
\end{theorem}
\begin{proof}
  The action of $D=H_1(\Sigma)$ on $\ab\Sigma$ can be described as
  follows: If $h\in D$ and $p\in \ab\Sigma$ and $\gamma\colon[0,1]\to
  \Sigma$ is any loop based at $\pi(p)$ in $\Sigma$ whose homology
  class is $h$, then the lift $\widetilde \gamma$ of $\gamma$ that
  ends at $p$ starts $hp$.

  For a regular value $\lambda$ of the function $x/|x|$ on $\Sigma-K$
  consider the set $A:=(x/|x|)^{-1}(\lambda)\cup K$. This set can be
  considered as a smooth $2$--chain in $\Sigma$ with boundary $dK$, so
  $\ell(h,[K])=\frac1d\,\gamma\cdot A$. Denote the inverse image of
  $A$ in $\ab\Sigma$ by $\widetilde A$. So
  $$\widetilde A=\{p\in\ab\Sigma~|~d\arg(x(p))=\arg(\lambda)\}\,.$$
  The intersection number $A\cdot\gamma$ equals the algebraic number
  of intersections of $\widetilde A$ with $\widetilde \gamma$. But the
  function $x$ changes continuously along $\widetilde \gamma$, with
  value at the point $hp$ of $\widetilde \gamma$ some power of
  $e^{2\pi i/d}$ times its value at $p$. The power in question is
  clearly, up to sign, the intersection number $\widetilde
  A\cdot\widetilde \gamma$, since, as one moves along
  $\widetilde\gamma$ from $p$ to $hp$ the change in argument of
  $x(\widetilde\gamma(t))$ can be followed by counting how this
  argument passes through values of the form $(\arg(\lambda)+2\pi
  k)/d$. Since $\widetilde \gamma$ is oriented from $hp$ to $p$, the
  sign is as stated in the theorem.
\end{proof}
\subsection{Milnor number and $\delta$-invariant}

Let $(X,o)$ be a normal surface singularity with $\Q$HS link
$\Sigma$ and $z\colon(X,o)\to (\C,0)$ a holomorphic germ which
vanishes with degree $d$ on its zero set $B$.\comment{WDN deleted:
Assume that $B$ is irreducible, and that one cannot write
$z=(z')^k$ for some $k>1$.}  By Lemma \ref{le:lift}, a $d$--th
root of $z$ lifts to a well-defined function $x\colon (V,0)\to
(\C,0)$ on the \UAC{} $(V,0)$ of $(X,o)$, which vanishes to degree
$1$ on its zero-set $C\subset V$. We compute the
$\delta$-invariant of $(C,o)$; this is a topological computation.

The link of the pair $(V,C)$ is a fibered \comment{JMW--make sure we
  are consistently ''fiber" vs. ''fibre". WDN done} link whose ``Milnor fiber''
$F$ is diffeomorphic to $F=x^{-1}(\delta)\cap D^{2N}_\epsilon$ for
some sufficiently small $0<\delta<\!\!<\epsilon<\!\!<1$, where
$D^{2N}_\epsilon$ is the $\epsilon$--ball about the origin for some
embedding $(V,0)\to (\C^N,0)$.  A standard formula
\begin{equation}
  \label{eq:delta}
\delta(C)=\frac12(r-\chi(F))
 \end{equation}
relates the  $\delta$--invariant of $C$ with the number of branches $r$ and
the Euler characteristic of its smoothing $F$ (\cite{buchweitz-greuel}).
We thus want to compute $\chi(F)$.

Let $\Gamma$ be the resolution graph for a simultaneous good
resolution of $(X,B,o)$ and $v_1$ the vertex corresponding to the
exceptional curve $E_{v_1}$ that the proper transform of $C$ meets. So $B$
meets $E_{v_1}$ transversally in one point and meets no other exceptional
curve.
\begin{lemma}
  \label{le:chi of fiber}
  With the above notation, $\chi(F)=-\nu(\Gamma,v_1)$ where
  $$\nu(\Gamma,v):= \sum_{v}(\delta'_v-2)\ell_{v_1v}\,,$$
  the sum is over all vertices $v$ of $\Gamma$, and
  $\delta'_v=\delta_v$ if $v\ne v_1$ and $\delta'_v=\delta_v+1$ if
  $v=v_1$ (where $\delta_v$ is valency of $v$). In particular, if
  $v_1$ is a leaf then the summand for $v_1$ is zero and, by
  \eqref{eq:delta},
  $$\delta(C)=\frac12\bigl(r+\nu(\Gamma,v_1)\bigr)=
  \frac12\bigl(r+\sum_{v\ne v_1}(\delta_v-2)\ell_{v_1v}\bigr)\,.$$
\end{lemma}
\begin{proof} If $dC+\sum a_vE_v$ is the zero set of $z$ on the
  resolution of $X$ then the intersection equations $(dC+\sum
  a_vE_v).E_w=0$ show that the order of vanishing $a_v$ of $z$ on a
  curve $E_v$ of the resolution is $d(\ell_{v_1v}/|D|)$.  It follows
  by a standard argument (originally due to A'Campo \cite{a'campo})
  that the Milnor fiber $F_z$ of $z$ has Euler characteristic
$$ \chi(F_z)=\frac d{|D|}\,\sum_v(2-\delta'_v) \ell_{v_1v}\,.$$
Let $X'$ and $z'$ be as in the proof of Lemma
  \ref{le:lift}. As there, we may assume $X'$ is connected. Then $d$
  is a divisor of $|D|$.

  \comment{JMW--check correctness of the last pair; also, we need to
    distinguish $B$ from the zero-set of $x$, since we have both $K$
    and $dK$. WDN ?? $B$
    and zero set of $x$ are in different spaces.}We can think of the
  Milnor fibers $F$ and $F_z$ as the fibers of the fibered
  (multi\hbox{-)}links $(\tilde\Sigma^{ab},L)$ and $(\Sigma, dK)$,
  which are the links of the pairs $(V,C)$ and $(X,B)$
  respectively. We have a diagram whose horizontal rows are Milnor
  fibrations and whose vertical arrows are covering maps:
$$\xymatrix{
F\ar[r]\ar[d]&\tilde\Sigma^{ab}-L\ar[r]\ar[d]&S^1\ar[d]\\
F_z\ar[r]&\Sigma-K\ar[r]&S^1}
$$
The degrees of the second and third vertical arrows are $|D|$ and $d$
respectively, so the first vertical arrow has degree $|D|/d$. Thus
$\chi(F)=|D|/d(\chi(F_z))$ and the lemma is proved.
\end{proof}

\subsection{Topological meaning of $\hat G$}

This subsection is a digression, without proofs, about the
topology underlying Section \ref{sec:curves from diagrams}
($D$-curves determined by rooted resolution diagrams). In that
section, after selecting a root leaf of the resolution diagram
$\Gamma$, the group $\hat G$ arose in terms of a $\C^*$--action
that is not easily seen to be part of the topological data. But
$\hat G$ has a very simple topological meaning.

If $\Sigma$ is the
link of our singularity and $K$ the knot corresponding to the root
leaf, then $\hat G = H_1(\Sigma_0)$, where $\Sigma_0$ is the knot
exterior (complement of an open solid torus neighborhood of $K$). The
homology class of a meridian curve $M$ of $K$ in $\partial\Sigma_0$
represents the element $\hat Q$ of Proposition \ref{prop:M}, while the
end-knots corresponding to the non-root leaves of $\Gamma$ represent
the elements $\overline\chi_j$ of Section \ref{sec:curves from
  diagrams}.

In particular, the value semigroup is the subsemigroup of
$H_1(\Sigma_0)$ generated by the classes of the end knots, and the
semigroup and congruence conditions together mean that the homology
class of the meridian of $K$ is a positive linear combination of the
homology classes of the end knots. 

Finally, in the induction of Section \ref{sec:curves from
diagrams}, the resolution sub-diagrams $\Gamma_i$ determine knot
exteriors $\Sigma_i$ which embed disjointly into $\Sigma_0$ (in an
obvious way) with complement $D_k\times S^1$, where $D_k$ is a
$k$--holed disk. The meridian curves $M_i$ for the $\Sigma_i$
match fibers of this $D_k\times S^1$, and Theorem \ref{th:ker}
describes a part of the homology exact sequence for the pair
$(\Sigma_0,\bigcup_{j=1}^k\Sigma_k)$.

\section{Corollaries and applications of the End Curve Theorem}
\label {sec:examples}

As corollaries of the End Curve Theorem, one can explain systematically
 why all previously known examples of splice quotients are in fact of this
 type.

\begin{corollary}[\cite{neumann83}]  The \UAC{} of a
weighted homogeneous singularity with $\Q HS$ link is a Brieskorn
complete intersection, and the covering group acts diagonally on the
 coordinates.
\end{corollary}
\begin{proof} The minimal resolution graph $\Gamma$ has one node, and
  the $t$ leaves correspond to the $\C^*$-orbits with non-trivial
  isotropy.  We will show that for every leaf of $\Gamma$, there
  exists a weighted homogeneous end-curve function.  Via the End Curve
  Theorem, weighted roots of these functions generate the maximal
  ideal of the \UAC; also, the defining equations begin with sums of
  admissible monomials, which are powers of these root functions.
  Because of quasihomogeneity, there can be no higher terms in these
  equations, so that the splice equations give a Brieskorn complete
  intersection.

  Let $X$ be the affine variety with $\C^*$-action. The $\Q$HS
  condition is equivalent to the fact that $X-\{0\}/\C^*$ is a
  rational curve. The $\C^*$-action on $X-\{0\}$ has finitely many
  orbits with non-trivial isotropy, and the closures of these orbits
  are end-curves. We shall show that these orbits---in fact,
  every orbit---can be cut out by a weighted homogeneous function on
  $X$. Consider a common multiple $N$ of the orders of the isotropy
  groups and let $\mu_N$ be the cyclic subgroup of $\C^*$ of order
  $N$. Then every non-trivial $\C^*$ orbit of $X'=X/\mu_N$ is
  isomorphic to $\C^*/\mu_N$, so $X'$ is a cone over a rational curve,
  thus a cyclic quotient singularity. It is thus easy to see that any
  non-trivial $\C^*$--orbit of $X'$ is cut out by some weighted
  homogeneous function $f$, and composing $f\colon X'\to \C$ with the
  projection $X\to X'$ then gives the desired function on $X$.
  \Omit{OLD VERSION: Let $X=$ Spec $A$ be the affine
  variety with $\C^*$-action, and $Y\rightarrow X$ the Seifert partial
  resolution, which has cyclic quotient singularities at $p_1,
  \cdots,p_t$ along the irreducible exceptional curve.  Let $N$ be the
  least common multiple of the orders of the non-trivial isotropy
  groups, and divide $X$ and $Y$ by the cyclic subgroup $\mu_N$ of
  $\C^*$ of order $N$.  Each formerly singular orbit of the $\C^*$
  action is now equivariantly equivalent to
  $(\C^*/\mu_{p_i})/\mu_N=\C^*/\mu_N$, so the quotient $X'=X/\mu_N$
  and is a cone over a smooth rational curve and and $Y'=Y/\mu_N$ its
  minimal resolution.  Further, the action is free on $Y$ except on
  the closure of the $t$ $\C^*$-orbits.  The images of these loci on
  $Y'$ are the $t$ fibers of the line bundle $Y'\rightarrow \mathbb
  P^1$ over the relevant $t$ points.  Since $X'$ is itself a cyclic
  quotient singularity, it is easy to see that each such fiber $F_i$
  is the reduced proper transform of the zeroes of a function $f_i$ on
  $X'$.  The zeroes of $f_i$ on $Y$ is thus the closure of a
  $\C^*$-orbit.  Composing $f_i$ with the map $X\to X'$ we have a
  function on $X$ that vanishes precisely on the $i$--th singular
  $\C^*$ orbit. Since these singular fibers are end-curves for $X$, we
  have constructed the desired end-curve functions.}
\end{proof}

\begin{corollary}[Okuma \cite{okuma2}]Rational singularities, and
  minimally elliptic singularities with $\Q HS$ link, are splice
  quotients.  In particular, their \UAC's are complete intersections.
\end{corollary}
\begin{proof} It was explained in Theorem 13.2 of
  \cite{nw1}why the End Curve Theorem would imply this result.
  Specifically, standard results on rational singularities
  easily produce end-curve functions (or more generally, functions
  satisfying any topologically allowed vanishing).  The existence of
  such functions in the minimally elliptic case is slightly harder,
  but is proved by M. Reid in \cite{reid}, Lemma, p. 122.
\end{proof}

We remark that Okuma's proof is different from ours, with a key step
the argument that the root functions generate the maximal ideal of the
\UAC.  It uses a strong condition satisfied by the graphs of rational
and minimally elliptic singularities, but not by splice-quotients in
general.  In \cite{neumann-wahl03}, the first non-trivial case of this
theorem was proved, showing that the \UAC{} of a quotient-cusp (the
simplest rational singularity whose graph has two nodes) is a complete
intersection cusp singularity.

\begin{corollary} [\cite{nw2}, Theorem 4.1] Let $(X,o)$ be a normal
  surface singularity with integral homology sphere link, for which
  all the knots associated to leaves are links of hypersurface
  sections.  Then the semigroup condition is fulfilled, and $X$ is a
  complete intersection of splice type.
\end{corollary}

``Equisingular" deformations of a splice quotient singularity whose
link is an \emph{integral} homology sphere should remain splice
quotients; this is definitely true for positive weight deformations of
weighted homogeneous singularities with $\Z HS$ links.  On the other
hand, it is not true even for fairly simple splice quotients
(cf. \cite{nemethi et al}), and this becomes clear via the End Curve
Theorem.  The following example comes from E. Sell's Ph.D. thesis.

\begin{example}[\cite{sell}, 3.1.4]\label{ex:sell} The weighted homogeneous
  singularity $X$ defined by $z^2=x^4+y^9$ has resolution dual graph
  and (reduced) splice diagram
    $$
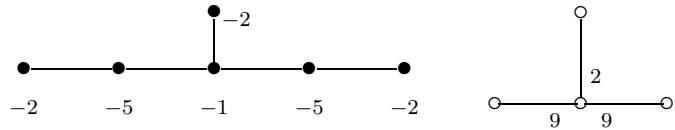
\xymatrix@R=8pt@C=30pt@M=0pt@W=0pt@H=0pt{
&&\righttag{\bullet}{-2}{6pt}\lineto[d]&&&&\Circ\\
&&\lineto[u]&\\
\undertag{\bullet}{-2}{4pt}\lineto[r]&\undertag{\bullet}{-5}{4pt}\lineto[r]&\undertag{\bullet}{-1}{4pt}\lineto[r]\lineto[u]&\undertag{\bullet}{-5}{4pt}\lineto[r]&\undertag{\bullet}{-2}{4pt}\\
&&&&&\Circ\lineto[r]_(.7)9&\Circ\lineto[r]_(.3)9\lineto[uuu]_(.3)2&\Circ\\\\
&&\\
}
$$
\comment{JMW--insert pictures, WDN Done}Rewriting the equation as $(z-x^2)(z+x^2)=y^9$, one
sees that the functions $z\pm x^2$ vanish 9 times along their
zero-sets; they, together with $x$ (whose zero-set is reduced) are
end-curve functions.  The discriminant group is cyclic of order 9.
So, one can form the \UAC{} by adjoining a $U$ satisfying $U^9=z-x^2$,
and a $V$ satisfying $V^9=z+x^2$, along with $x$.  Then $(UV)^9=y^9$;
but since the \UAC{} is a normal domain, one must have (perhaps changing
$V$ by a $9$--th root of 1) that $y=UV$.  Thus, the \UAC{} is the
hypersurface $V^9=U^9+2x^2$, with discriminant group
action $$(U,V,x)\mapsto (\zeta U,\zeta ^{-1}V,x),$$ where $\zeta$ is a
primitive $9$--th root of 1.  The versal deformation of weight $\geq 0$ of
the \UAC{} that is equivariant with respect to the group action is smooth
of dimension 3, and defined by
$$V^9=U^9+2x^2+t_1(UV)^5+t_2(UV)^6+t_3(UV)^7.$$
Taking invariants to obtain the ``versal splice quotient deformation"
of $X$, and changing coordinates, gives the family
$$y^9=(z-x^2-\frac{1}{2}(t_1y^5+t_2y^6+t_3y^7))(z+x^2+\frac{1}{2}(t_1y^5+t_2y^6+t_3y^7)),$$
which can be written (to show the deformation of the curve $x^4+y^9=0$)
$$z^2=x^4+y^9+t_1x^2y^5+t_2x^2y^6+t_3x^2y^7+\text{quadratic terms in $t_i$'s}.$$
One important point is that the first equation for the versal splice
quotient deformation shows explicitly how to lift the end-curve
functions $z\pm x^2$ under deformation so that \emph{they remain
  end-curve functions}, i.e., continue to vanish to order 9 along
their zero-sets.  A second point is that the positive weight
deformation $z^2=x^4+y^9+txy^7$ is \emph{not} a splice-quotient
deformation, even to first order.
\end{example}

\begin{remark} More generally, in her thesis \cite{sell} E. Sell
  considers singularities $z^n=f(x,y)$ with $\Q HS$ link and $f$
  analytically irreducible.  The splice and semigroup conditions
  depend only on $n$ and the topological type (i.e., Puiseux pairs) of
  $f$, and it turns out that very rarely are they satisfied.  However,
  in all cases where they are satisfied, there exist special $f$ so
  that the singularity is a splice quotient; and in these cases the
  defining equation can be rewritten (as in the example above) to
  highlight the end curve functions.  As above, only deformations
  which are both equisingular and preserve the order of vanishing of
  these functions give splice quotient deformations.  \end{remark}

\end{document}